\def\DateTime{4/February/2009, 8:30(JP)}
\def\Version{Version $1.0\beta$}
\def\yes{\if00}
\def\no{\if01}
\def\iftwelvept{\no}

\def\ifusepdf{\no}
\def\ifpsfont{\yes}

\iftwelvept
\documentclass[leqno,12pt]{amsart}
\else
\documentclass[leqno]{amsart}
\fi
\usepackage{amssymb}
\usepackage{amscd}
\usepackage{latexsym}
\usepackage{verbatim}
\usepackage[all]{xy}

\ifusepdf
\usepackage{hyperref}
\else\fi
\ifpsfont
\usepackage[T1]{fontenc}
\usepackage{times}
\else\fi

\iftwelvept
\else\fi

\theoremstyle{plain}
\newtheorem{Theorem}{Theorem}[section]
\newtheorem{Proposition}[Theorem]{Proposition}
\newtheorem{Lemma}[Theorem]{Lemma}
\newtheorem{Corollary}[Theorem]{Corollary}

\newtheorem{Claim}{Claim}[Theorem]

\theoremstyle{definition}

\def\rom{\textup}
\newcommand{\ZZ}{{\mathbb{Z}}}
\newcommand{\QQ}{{\mathbb{Q}}}
\newcommand{\RR}{{\mathbb{R}}}
\newcommand{\CC}{{\mathbb{C}}}

\newcommand{\KK}{{\mathbb{K}}}

\newcommand{\OO}{{\mathcal{O}}}

\newcommand{\Image}{\operatorname{Image}}
\newcommand{\Ker}{\operatorname{Ker}}
\newcommand{\Coker}{\operatorname{Coker}}
\newcommand{\Spec}{\operatorname{Spec}}

\newcommand{\Supp}{\operatorname{Supp}}
\newcommand{\codim}{\operatorname{codim}}

\newcommand{\rank}{\operatorname{rk}}
\newcommand{\acherncl}{\widehat{{c}}}
\newcommand{\aPic}{\widehat{\operatorname{Pic}}}
\newcommand{\adeg}{\widehat{\operatorname{deg}}}
\newcommand{\avol}{\widehat{\operatorname{vol}}}

\newcommand{\aH}{\hat{H}^0}

\newcommand{\vol}{\operatorname{vol}}

\newcommand{\sub}{\operatorname{sub}}
\newcommand{\quot}{\operatorname{quot}}

\newcommand{\Sym}{\operatorname{Sym}}
\newcommand{\aBs}{\operatorname{Bs}}
\newcommand{\aSBs}{\operatorname{SBs}}
\newcommand{\Cone}{\operatorname{Cone}}
\newcommand{\Conv}{\operatorname{Conv}}
\newcommand{\CL}{\operatorname{CL}}
\newcommand{\aAmp}{\operatorname{\widehat{Amp}}}
\newcommand{\aBig}{\operatorname{\widehat{Big}}}
\newcommand{\aEff}{\operatorname{\widehat{Eff}}}
\newcommand{\QSat}{\operatorname{Sat}}
\newcommand{\ord}{\operatorname{ord}}

\newcommand{\CQED}{{\unskip\nobreak\hfil\penalty50\quad\null\nobreak\hfil
{$\Box$}\parfillskip0pt\finalhyphendemerits0\par\medskip}}
\newcommand{\rest}[2]{\left.{#1}\right\vert_{{#2}}}

\newcommand{\FF}{{\mathbb{F}}}

\begin{document}

\title[Estimation of arithmetic linear series]%
{Estimation of arithmetic linear series 
}
\author{Atsushi Moriwaki}
\address{Department of Mathematics, Faculty of Science,
Kyoto University, Kyoto, 606-8502, Japan}
\email{moriwaki@math.kyoto-u.ac.jp}
\date{\DateTime, (\Version)}

\maketitle

\renewcommand{\theTheorem}{\Alph{Theorem}}

\section*{Introduction}

In the paper \cite{LazMus},
Lazarsfeld and Musta\c{t}\u{a} propose general and systematic usage of Okounkov's idea in order to study
asymptotic behavior of linear series on an algebraic variety.
It is a very simple way, but it yields a lot of consequences, like Fujita's approximation theorem.
Yuan \cite{YuanVol} generalized this way to the arithmetic situation, and
he established the arithmetic Fujita's approximation theorem, which was also proved by Chen \cite{HChenFujita} independently.
In this paper, we introduce arithmetic linear series and give a general way to estimate them based on Yuan's idea.
As an application, we consider an arithmetic analogue of the algebraic restricted volumes.

\subsection*{Arithmetic linear series}
Let $X$ be a $d$-dimensional projective arithmetic variety and
$\overline{L}$ a continuous hermitian invertible sheaf on $X$.
Let $K$ be a symmetric convex lattice in $H^0(X, L)$ with
\[
K \subseteq B_{\sup}(\overline{L}) := \{ s \in H^0(X, L)_{\RR} \mid \Vert s \Vert_{\sup} \leq 1\},
\]
that is, $K$ satisfies the following properties (for details, see Subsection~\ref{subsec:convex:lattice}):
\begin{enumerate}
\renewcommand{\labelenumi}{\rom{(\arabic{enumi})}}
\item
$K = \{ x \in  \langle K \rangle_{\ZZ} \mid \exists m \in \ZZ_{>0} \ \ mx \in m \ast K \}$,
where $\langle K \rangle_{\ZZ}$ is the $\ZZ$-submodule generated by $K$ and
$m \ast K = \{ x_1 + \cdots + x_m \mid x_1, \ldots, x_m \in K\}$.

\item
$-x \in K$ for all $x \in K$.

\item
$K \subseteq H^0(X, L) \cap B_{\sup}(\overline{L})$.
\end{enumerate}
We call $K$ an {\em arithmetic linear series} of $\overline{L}$.
In the case where $K = B_{\sup}(\overline{L}) \cap H^0(X, L)$,
it is said to be {\em complete}.
One of main results of this paper is a uniform estimation of
the number of the arithmetic linear series in terms of the number of valuation vectors.

\begin{Theorem}
Let $\nu$ be the valuation attached to a good flag 
over a prime $p$ \rom{(}see Subsection~\rom{\ref{subsection:good:flag:over:prime}} for the definition of a good flag over a prime\rom{)}.
Then we have
\begin{multline*}
\vert \#\nu(K \setminus \{ 0\}) \log p - \log \#(K) \vert \\
\leq \left( \log\left(4 p \rank \langle K \rangle_{\ZZ} \right) + 
\frac{\sigma(\overline{L}) + \log\left(2 p \rank \langle K \rangle_{\ZZ} \right)}{\log p}  \log(4) \rank H^0(\OO_X)\right)
\rank \langle K \rangle_{\ZZ},
\end{multline*}
where $\sigma(\overline{L})$ is given by
\[
\sigma(\overline{L}) := \inf_{\text{$\overline{A}$ : ample}}
\frac{\adeg (\acherncl_1(\overline{A})^{d-1} \cdot \acherncl_1(\overline{L}))}{\deg (A_{\QQ}^{d-1})}.
\]
\end{Theorem}

An ideal for the proof of the above theorem is essentially same as one of Yuan's paper \cite{YuanVol},
in which he treated only the complete arithmetic linear series in my sense.
A new point is the usage of convex lattices, that is, a general observation for arithmetic linear series.
By this consideration, we obtain several advantages in applications.
For example, we have the following theorem, which is a stronger version of \cite[Theorem~3.3]{YuanVol}.
The arithmetic Fujita's approximation theorem is
an immediate consequence of it.

\begin{Theorem}
Let $\overline{L}$ be a big continuous hermitian invertible sheaf on $X$.
For any positive $\epsilon$, there is a positive integer $n_0 = n_0(\epsilon)$ such that,
for all $n \geq n_0$,
\[
\liminf_{k\to\infty} \frac{
\log\#\CL( V_{k,n} )}{n^dk^d}
\geq \frac{\avol(\overline{L})}{d!} - \epsilon,
\]
where $V_{k,n} = \{ s_1 \otimes \cdots \otimes s_k  \in H^0(X, knL) \mid s_1, \ldots, s_k \in \aH(X, n \overline{L})\}$ and $\CL( V_{k,n} )$ is 
the convex lattice hull of $V_{k,n}$, that is,
\[
\CL( V_{k,n} ) = \{ x \in  \langle V_{k,n} \rangle_{\ZZ} \subseteq H^0(X, knL)  \mid 
\exists m \in \ZZ_{>0} \ \ mx \in m \ast V_{k,n} \}
\]
\rom{(}cf. for details, see Subsection~\rom{\ref{subsec:convex:lattice}}).
\end{Theorem}

\subsection*{Arithmetic analogue of restricted volume}
For further applications, let us consider an arithmetic analogue of the restricted volume.
Let  $Y$ be a $d'$-dimensional arithmetic subvariety of $X$, that is, $Y$ is an integral subscheme of $X$ such that $Y$ is flat over $\Spec(\ZZ)$.
Let $\overline{L}$ be a continuous hermitian invertible sheaf on $X$.
We denote $\Image(H^0(X, L) \to H^0(Y, \rest{L}{Y}))$ by $H^0(X|Y, L)$.
Let $\Vert\cdot\Vert_{\sup,\quot}^{X|Y}$ be the the quotient norm of  $H^0(X|Y, L) \otimes_{\ZZ} \RR$ induced by
the surjective homomorphism 
\[
H^0(X, L) \otimes_{\ZZ} \RR \to H^0(X|Y, L) \otimes_{\ZZ} \RR
\]
and
the norm $\Vert\cdot\Vert_{\sup}$ on $H^0(X, L) \otimes_{\ZZ} \RR$. 
We define $\aH_{\quot}(X|Y, \overline{L})$ and $\avol_{\quot}(X|Y, \overline{L})$ to be
\begin{align*}
\aH_{\quot}(X|Y, \overline{L}) & := \left\{ s \in H^0(X|Y, L) \mid \Vert s \Vert_{\sup,\quot}^{X|Y} \leq 1 \right\}\quad\text{and}\\
\avol_{\quot}(X|Y, \overline{L}) & := \limsup_{m\to\infty} \frac{\log \# \aH_{\quot}(X|Y, m\overline{L})}{m^{d'}/d'!}.
\end{align*}
Note that $\aH_{\quot}(X|Y, \overline{L})$ is an arithmetic linear series of $\rest{\overline{L}}{Y}$.
A continuous hermitian invertible sheaf $\overline{L}$ is said to be {\em $Y$-effective}
if there is $s \in \aH(X, \overline{L})$ with $\rest{s}{Y} \not= 0$.
Moreover, $\overline{L}$ is said to be {\em $Y$-big} if
there are a positive integer $n$,
an ample $C^{\infty}$-hermitian invertible sheaf $\overline{A}$ and
a $Y$-effective continuous hermitian invertible sheaf $\overline{M}$ such that
$n \overline{L} = \overline{A} + \overline{M}$.
The semigroup consisting of isomorphism  classes of $Y$-big continuous hermitian invertible sheaves
is denoted by $\aBig(X;Y)$.
Then we have the following theorem, which is a generalization of \cite{ChenBig} and \cite[Theorem~2.7 and Theorem~B]{YuanVol}.

\begin{Theorem}
\begin{enumerate}
\renewcommand{\labelenumi}{\rom{(\arabic{enumi})}}
\item
If $\overline{L}$ is a $Y$-big continuous hermitian invertible sheaf on $X$, then
$\avol_{\quot}\left(X|Y, \overline{L}\right) > 0$ and
\[
\avol_{\quot}\left(X|Y, \overline{L}\right) = \lim_{m\to\infty} \frac{\log \# \aH_{\quot}(X|Y, m\overline{L})}{m^{d'}/d'!}.
\]
In particular, $\avol_{\quot}(X|Y, n \overline{L}) = n^{d'} \avol_{\quot}(X|Y, \overline{L})$.
\item
$\avol_{\quot}\left(X|Y, -\right)^{\frac{1}{d'}}$ is concave on $\aBig(X;Y)$, that is,
\[
\hspace{4em}
\avol_{\quot}\left(X|Y, \overline{L}+\overline{M}\right)^{\frac{1}{d'}}
\geq \avol_{\quot}\left(X|Y, \overline{L}\right)^{\frac{1}{d'}} +
\avol_{\quot}\left(X|Y, \overline{M}\right)^{\frac{1}{d'}}.
\]
holds for any 
$Y$-big continuous hermitian invertible sheaves $\overline{L}$ and $\overline{M}$ on $X$.

\item
If $\overline{L}$ is a $Y$-big continuous hermitian invertible sheaf on $X$, then,
for any positive number $\epsilon$, there is a positive integer $n_0 = n_0(\epsilon)$ such that,
for all $n \geq n_0$,
\begin{multline*}
\hspace{4em}
\liminf_{k\to\infty} \frac{
\log\#\CL\left( \{ s_1 \otimes \cdots \otimes s_k \mid s_1, \ldots, s_k \in \aH_{\quot}(X|Y, n \overline{L})\} \right)}{n^{d'}k^{d'}} \\
\geq \frac{\avol_{\quot}(X|Y, \overline{L})}{d'!} - \epsilon,
\end{multline*}
where the convex lattice hull is considered in $H^0(X|Y, knL)$.

\item
If $X_{\QQ}$ is smooth over $\QQ$ and $\overline{A}$ is an ample $C^{\infty}$-hermitian invertible sheaf on $X$, then
\begin{align*}
\hspace{4em}
\avol_{\quot}\left(X|Y, \overline{A}\right) & = \avol(Y, \rest{\overline{A}}{Y}) \\
& =
\lim_{m\to\infty} \frac{\log \# \Image(\aH(X, m\overline{A}) \to H^0(X|Y, mA))}{m^{d'}/d'!}.
\end{align*}
\end{enumerate}
\end{Theorem}

Let $C^0(X)$ be the set of real valued continuous functions $f$ on $X(\CC)$ 
such that $f$ is invariant under  the complex conjugation map on $X(\CC)$.
We denote the group of isomorphism classes of continuous
hermitian invertible sheaves on $X$ by $\aPic(X;C^0)$.
Let $\overline{\OO} : C^0(X) \to \aPic(X;C^0)$ be the homomorphism given by
\[
\overline{\OO}(f) = (\OO_X, \exp(-f)\vert\cdot\vert_{can}).
\]
$\aPic_{\RR}(X;C^0)$ is defined to be 
\[
\aPic_{\RR}(X;C^0) := \frac{\aPic(X;C^0) \otimes \RR}{\left\{ \sum_i \overline{\OO}(f_i) \otimes x_i \mid 
f_i \in C^0(X), x_i \in \RR \ (\forall i),\  \sum_i x_if_i = 0 \right\}}.
\]
Let $\gamma : \aPic(X;C^0) \to \aPic_{\RR}(X;C^0)$ be the natural homomorphism given by 
the composition of homomorphisms
\[
\aPic(X;C^0) \to
\aPic(X;C^0)  \otimes \RR \to \aPic_{\RR}(X;C^0).
\]
Let  $\aBig_{\RR}(X;Y)$ be the cone in $\aPic_{\RR}(X;C^0)$ generated by
$\{ \gamma(\overline{L}) \mid \overline{L} \in \aBig(X;Y) \}$.
Note that $\aBig_{\RR}(X;Y)$ is an open set in $\aPic_{\RR}(X;C^0)$ in the strong topology, that is,
$\aBig_{\RR}(X;Y) \cap W$ is an open set in $W$ in the usual topology for any finite dimensional vector subspace $W$ of $\aPic_{\RR}(X;C^0)$.
The next theorem guarantees that 
\[
\avol_{\quot}(X|Y, -) : \aBig(X;Y) \to \RR
\]
extends to a continuous function $\avol''_{\quot}(X|Y,-) : \aBig_{\RR}(X;Y) \to \RR$,
which can be considered as a partial generalization of \cite{MoCont} and \cite{MoExt}.

\begin{Theorem}
There is a unique positive valued continuous function 
\[
\avol''_{\quot}(X|Y,-) : \aBig_{\RR}(X;Y) \to \RR
\]
with the following properties:
\begin{enumerate}
\renewcommand{\labelenumi}{\rom{(\arabic{enumi})}}
\item
The following diagram is commutative:
\[
\xymatrix{
\aBig(X;Y)  \ar[rr]^(.57){\avol_{\quot}(X|Y,-)} \ar[d]_{\gamma} & & \RR \\
\aBig_{\RR}(X;Y) \ar[rru]_(.5){\ \ \avol''_{\quot}(X|Y-)} & \\
}
\]

\item
$\avol''_{\quot}(X|Y,-)^{\frac{1}{d'}}$ is positively homogeneous and concave on
$\aBig_{\RR}(X;Y)$, that is,
\[
\hspace{3em}
\begin{cases}
\avol''_{\quot}(X|Y,\lambda x)^{\frac{1}{d'}} = \lambda \avol''_{\quot}(X|Y,x)^{\frac{1}{d'}} \\
\avol''_{\quot}\left(X|Y, x + y\right)^{\frac{1}{d'}}
\geq \avol''_{\quot}\left(X|Y, x\right)^{\frac{1}{d'}} +
\avol''_{\quot}\left(X|Y, y\right)^{\frac{1}{d'}}
\end{cases}
\]
hold
for all $\lambda \in \RR_{>0}$ and $x, y \in \aBig_{\RR}(X;Y)$.

\end{enumerate}
\end{Theorem}

\subsection*{Acknowledgements}
I would like to thank Prof. Yuan for communications.

%
%

\renewcommand{\thesubsubsection}{\arabic{subsubsection}}

\bigskip
\renewcommand{\theequation}{CT.\arabic{subsubsection}.\arabic{Claim}}
\subsection*{Conventions and terminology}
We fix several conventions and terminology of this paper.

\subsubsection{}
\label{CT:sub:semigroup:monoid}
Let $M$ be a $\ZZ$-module, and let $A$ be a {\em sub-semigroup} of $A$, that is,
$x + y \in A$ holds for all $x, y \in A$. If $0 \in A$, then
$A$ is called a {\em sub-monoid} of $M$.
The {\em saturation} $\QSat(A)$ of $A$ in $M$ is defined by
\[
\QSat(A) := \{ x \in M \mid \text{$n x \in A$ for some positive integer $n$} \}.
\]
It is easy to see that $\QSat(A)$ is a sub-semigroup of $M$.
If $A = \QSat(A)$, then $A$ is said to be {\em saturated} .

\subsubsection{}
\label{CT:convex:set}
Let $\KK$ be either $\QQ$ or $\RR$, and let $V$ be a vector space over $\KK$.
A subset $C$ of $V$ is called a convex set in $V$ if
$t x + (1-t)y \in C$ for all $x, y \in C$ and $t \in \KK$ with $0\leq t \leq 1$.
For a subset $S$ of $V$, it is easy to see that
the subset
\[
\{ t_1 s_1 + \cdots + t_r s_r \mid s_1, \ldots, s_r \in S,\ t_1, \ldots, t_r \in \KK_{\geq 0},\ 
t_1 + \cdots + t_r = 1 \}
\]
is a convex set. It is called the {\em convex hull by $S$} and is denoted by
$\Conv_{\KK}(S)$.
Note that $\Conv_{\KK}(S)$ is the smallest convex set containing $S$.
A function $f : C  \to \RR$ on a convex set $C$ is said to be  {\em concave over $\KK$} if
$f(t x + (1-t) y) \geq t f(x) + (1-t)f(y)$ holds for any $x, y \in C$ and $t \in \KK$ with $0 \leq t \leq 1$.

\subsubsection{}
\label{CT:cone}
Let $\KK$ and $V$ be the same as in the above \ref{CT:convex:set}.
A subset $C$ of $V$ is called  a {\em cone} in $V$ if the following conditions are satisfied:
\begin{enumerate}
\renewcommand{\labelenumi}{\rom{(\alph{enumi})}}
\item
$x + y \in C$ for any $x, y \in C$.

\item
$\lambda x \in C$ for any $x \in C$ and $\lambda \in \KK_{> 0}$.
\end{enumerate}
Note that a cone is a sub-semigroup of $V$.
Let $S$ be a subset of $V$. The smallest cone containing $S$, that is,
\[
\{ \lambda_1 a_1 + \cdots + \lambda_r a_r \mid a_1, \ldots, a_r \in S,\ \lambda_1, \ldots, \lambda_r \in \KK_{> 0} \}
\]
is denoted by $\Cone_{\KK}(S)$. It is called the {\em cone generated by $S$}.

\subsubsection{}
\label{CT:strong:topology}
Let $\KK$ and $V$ be the same as in the above \ref{CT:convex:set}.
The strong topology on $V$ means that  a subset $U$ of $V$ is open set in this topology if and only if,
 for any finite dimensional vector subspace $W$ of $V$ over $\KK$,
$U \cap W$ is open in $W$ in the usual topology.

It is easy to see that a linear map of vector spaces over $\KK$ is continuous in the strong topology.
Moreover,
a surjective linear map of vector spaces over $\KK$ is an open map in the strong topology.
In fact,
let 
$f : V \to V'$ be a surjective homomorphism of vector spaces over $\KK$,
$U$ an open set of $V$, and
$W'$ a finite dimensional vector subspace of $V'$ over $\KK$.
Then we can find a vector subspace $W$ of $V$ over $\KK$ such that $f$ induces
the isomorphism $\rest{f}{W} : W \to W'$.
If we set $\widetilde{U} = \bigcup_{t \in \Ker(f)} ( U + t)$, then
$\widetilde{U}$ is open and $f(W \cap \widetilde{U}) = W' \cap f(U)$, as required.

Let $V'$ be a vector subspace of $V$ over $\KK$.
Then the induced topology of $V'$ from $V$coincides with the strong topology of $V'$.
Indeed, let $U'$ be an open set of $V'$ in the strong topology. We can easily construct a linear map
$f : V \to V'$ such that $V' \hookrightarrow V \overset{f}{\longrightarrow} V'$ is the identity map.
Thus $f^{-1}(U')$ is an open set in $V$, and hence $U' = \rest{f^{-1}(U')}{V'}$ is an open set in the induced topology.

\subsubsection{}
\label{CT:arith:subvariety}
A closed integral subscheme of an arithmetic variety is called an {\em arithmetic subvariety} if
it is flat over $\Spec(\ZZ)$.

\subsubsection{}
\label{CT:Pic:cont}
Let $X$ be an arithmetic variety.
We denote the group of isomorphism classes of continuous hermitian (resp. $C^{\infty}$-hermitian) invertible sheaves by
$\aPic(X;C^0)$ (resp. $\aPic(X;C^{\infty})$). $\aPic(X;C^{\infty})$ is often denoted by $\aPic(X)$ for simplicity.
An element of $\aPic_{\QQ}(X;C^0) := \aPic(X;C^0) \otimes_{\ZZ} \QQ$ (resp. $\aPic_{\QQ}(X;C^{\infty}) := \aPic(X;C^{\infty}) \otimes_{\ZZ} \QQ$)
is called a continuous hermitian (resp. $C^{\infty}$-hermitian) $\QQ$-invertible sheaf.

\subsubsection{}
\label{CT:Arakelov:positivity}
A $C^{\infty}$-hermitian invertible sheaf $\overline{A}$ on a projective arithmetic variety $X$
is said to be {\em ample} if
$A$ is ample on $X$, the first Chern form $c_1(\overline{A})$ is positive on $X(\CC)$ and,
for a sufficiently large integer $n$,
$H^0(X, nA)$ is generated by the set
\[
\{ s \in H^0(X, nA) \mid \Vert s \Vert_{\sup} < 1\}
\]
as a $\ZZ$-module.
Note that, for $\overline{A}, \overline{L} \in \aPic(X;C^{\infty})$,
if $\overline{A}$ is ample, then there is a positive integer $m$ such that
$m \overline{A} + \overline{L}$ is ample.

\subsubsection{}
\label{CT:cont:herm:sup:ball}
Let $\overline{L}$ be a continuous hermitian invertible sheaf on a projective arithmetic variety $X$.
Then $B_{\sup}(\overline{L})$ is defined to be
\[
B_{\sup}(\overline{L}) = \{ s \in H^0(X, L)_{\RR} \mid \Vert s \Vert_{\sup} \leq 1 \}.
\]
Note that $\aH(X, \overline{L}) = H^0(X, L) \cap B_{\sup}(\overline{L})$.

\subsubsection{}
\label{C:T:valuation:vector}
Let $A$ be a noetherian integral domain and $t \not\in A^{\times}$.
As $\bigcap_{n \geq 0} t^n A =  \{ 0 \}$,
for $a \in A \setminus \{ 0\}$,  we can define $\ord_{tA}(a)$
to be
\[
\ord_{tA}(a) = \max \{ n \in \ZZ_{\geq 0} \mid a \in t^n A \}.
\]
Let $\{ 0 \} = P_0 \subsetneq P_1 \subsetneq \cdots \subsetneq P_d$ be
a chain of prime ideals of $A$.
Let $A_i = A/P_i$ for $i=0, \ldots, d$, and
let $\rho_{i} : A_{i-1} \to A_{i}$ be natural homomorphisms as follows.
\[
A = A_0 \overset{\rho_1}{\longrightarrow} A_1 \overset{\rho_2}{\longrightarrow} \cdots
\overset{\rho_{d-1}}{\longrightarrow} A_{d-1} \overset{\rho_{d}}{\longrightarrow}  A_{d}.
\]
We assume that
$P_d$ is a maximal ideal, and that
$P_{i}A_{i-1} = \Ker(\rho_{i})$ is a principal ideal of $A_{i-1}$ for every $i=1, \ldots, d$, that is,
there is $t_{i} \in A_{i-1}$ with $P_{i}A_{i-1} = t_{i} A_{i-1}$.
For $a \not= 0$, the {\em valuation vector} $(\nu_1(a), \ldots, \nu_{d}(a))$ of $a$
is defined in the following way:
\[
a_1 := a\quad\text{and}\quad\nu_1(a) := \ord_{t_1 A_0} (a_1).
\]
If $a_1 \in A_0, a_2 \in A_1, \ldots, a_{i} \in A_{i-1}$ and $\nu_1(a), \ldots, \nu_{i}(a) \in \ZZ_{\geq 0}$ are given,
then 
\[
a_{i+1} := \rho_{i}(a_{i}t_i^{-\nu_{i}(a)})\quad\text{and}\quad
\nu_{i+1}(a) := \ord_{t_{i+1}A_{i}}(a_{i+1}).
\]
Note that the valuation vector $(\nu_1(a), \ldots, \nu_{d}(a))$ does not depend on the choice of
$t_1, \ldots, t_d$.

Let $X$ be a noetherian integral scheme and
\[
Y_{\cdot} : Y_0 = X \supset Y_1 \supset Y_2 \supset \cdots \supset Y_d
\]
a chain of integral subschemes of $X$.
We say $Y_{\cdot}$ a {\em flag} if
$Y_d$ consists of a closed point $y$ and
$Y_{i+1}$ is locally principal at $y$ in $Y_{i}$ for all $i =0, \ldots, d-1$.
Let $A = \OO_{X, y}$ and $P_i$ the defining prime ideal of $Y_i$ in $A$.
Then we have a chain $P_0 \subsetneq P_1 \subsetneq \cdots \subsetneq P_d$ of prime ideals as above, so that 
we obtain the valuation vector $(\nu_1(a), \ldots, \nu_d(a))$
for each $a \in A \setminus \{ 0\}$.
It is called the {\em valuation vector} attached to the flag
$Y_{\cdot}$ and is denoted by $\nu_{Y_{\cdot}}(a)$ or $\nu(a)$.
Let $L$ be an invertible sheaf on $X$ and
$\omega$ a local basis of $L$ at $y$. Then,
for each $s \in H^0(X, L)$, we can find $a_s \in A$ with $s = a_s \omega$.
Then $\nu_{Y_{\cdot}}(a_s)$ is denoted by $\nu_{Y_{\cdot}}(s)$.
Note that $\nu_{Y_{\cdot}}(s)$ does not depend on the choice of $\omega$.

\renewcommand{\theTheorem}{\arabic{section}.\arabic{subsection}.\arabic{Theorem}}
\renewcommand{\theClaim}{\arabic{section}.\arabic{subsection}.\arabic{Theorem}.\arabic{Claim}}
\renewcommand{\theequation}{\arabic{section}.\arabic{subsection}.\arabic{Theorem}.\arabic{Claim}}

\section{Preliminaries}
\setcounter{Theorem}{0}
\subsection{Open cones}
Let $\KK$ be either $\QQ$ or $\RR$, and let $V$ be a vector space over $\KK$.
A cone in $V$ is said to be {\em open} if it is an open set in $V$ in the strong topology
(see Conventions and terminology~\ref{CT:strong:topology}).

\begin{Proposition}
\label{prop:big:open}
Let $C$ be a cone in $V$. Then we have the following:
\begin{enumerate}
\renewcommand{\labelenumi}{\rom{(\arabic{enumi})}}
\item
$C$ is open if and only if, for any $a \in C$ and $x \in V$, there is $\delta_0 \in \KK_{>0}$ such that
$a + \delta_0 x \in C$.

\item
Let $f : V \to V'$ be a surjective homomorphism of vector spaces over $\KK$.
\begin{enumerate}
\renewcommand{\labelenumii}{\rom{(\arabic{enumi}.\arabic{enumii})}}
\item
If $C$ is open in $V$, then $f(C)$ is also open in $V'$.

\item
If $C + \Ker(f) \subseteq C$,
then $f^{-1}(f(C)) =  C$.
\end{enumerate}
\end{enumerate}
\end{Proposition}

\begin{proof}
(1) If $C$ is open, then the condition in (1) is obviously satisfied.
Conversely we assume that, for any $a \in C$ and $x \in V$, there is $\delta_0 \in \KK_{>0}$ such that
$a + \delta_0 x \in C$. First let see the following claim:

\begin{Claim}
For any $a \in C$ and $x \in V$, there is $\delta_0 \in \KK_{>0}$ such that
$a + \delta x \in C$ holds for all $\delta \in \KK$ with $\vert \delta \vert \leq \delta_0$.
\end{Claim}

By our assumption, there are $\delta_1,\delta_2 \in \KK_{>0}$ such that
$a + \delta_1 x, a + \delta_2 (-x) \in C$.
For $\delta \in \KK$ with $-\delta_2 \leq \delta \leq \delta_1$, if we set
$\lambda = (\delta + \delta_1)/(\delta_1 + \delta_2)$, then
$0 \leq \lambda \leq 1$ and $\delta = \lambda \delta_1 + (1-\lambda)(-\delta_2)$.
Thus
\[
\lambda(b + \delta_1x) + (1-\lambda) (b + \delta_2(-x)) = b + \delta x \in C.
\]
Therefore, if we put $\delta_0 = \min \{ \delta_1, \delta_2 \}$, then the assertion of the claim follows.
\CQED

\medskip
Let $W$ be a finite dimensional vector subspace of $V$ over $\KK$ and
$a \in W \cap C$.
Let $e_1, \ldots, e_n$ be a basis of $W$.
Then, by the above claim, there is $\delta_0 \in \KK_{>0}$ such that
$a/n + \delta e_i \in C$ holds for all $i$ and all $\delta \in \KK$ with $\vert \delta \vert \leq \delta_0$.
We set 
\[
U = \{ x_1 e_1 + \cdots + x_n e_n \mid \vert x_1 \vert < \delta_0, \ldots, \vert x_n \vert < \delta_0 \}.
\]
It is sufficient to see that $a + U \subseteq C$. Indeed, if $x = x_1 e_1 + \cdots + x_n e_n \in U$,
then 
\[
a + x = \sum_{i=1}^n (a/n + x_i e_i) \in C.
\]

\medskip
(2) (2.1) follows from the fact that $f$ is an open map (cf. Conventions and terminology~\ref{CT:strong:topology}).
Clearly $f^{-1}(f(C)) \supseteq  C$. Conversely let $x \in f^{-1}(f(C))$.
Then there are $a \in C$ with $f(x) = f(a)$.
Thus we can find $u \in \Ker(f)$ such that $x - a  = u$ because $f(x - a) = 0$.
Hence
\[
x = a + u \in C + \Ker(f) \subseteq C.
\]
\end{proof}

To proceed with further arguments, we need the following two lemmas.

\begin{Lemma}
\label{lem:cone:sum}
Let $S$ and $T$ be subsets of $V$.
Then 
\[
\Cone_{\KK}(S + T) \subseteq \Cone_{\KK}(S) + \Cone_{\KK}(T),
\]
where $S + T = \{ s + t \mid s \in S, \ t \in T\}$.
Moreover, if $a t \in T$ holds for all $t \in T$ and
$a \in \ZZ_{\geq 0}$, then $\Cone_{\KK}(S + T) = \Cone_{\KK}(S) + \Cone_{\KK}(T)$.
\end{Lemma}

\begin{proof}
The first assertion is obvious.
Let $x \in \Cone_{\KK}(S) + \Cone_{\KK}(T)$.
Then there are $s_1, \ldots, s_r \in S$, $t_1, \ldots, t_{r'} \in T$,
$\lambda_1, \ldots, \lambda_r \in \KK_{> 0}$ and
$\mu_1, \ldots, \mu_{r'} \in \KK_{>0}$ such that
$x = \lambda_1 s_1 + \cdots + \lambda_r s_{r} + \mu_1 t_1 + \cdots + \mu_{r'} t_{r'}$.
We choose a positive integer $N$ with $N \lambda_1 > \mu_1 + \cdots + \mu_{r'}$. Then
\begin{multline*}
x = \left( \lambda_1 - \frac{\mu_1 + \cdots + \mu_{r'}}{N} \right) (s_1 + 0)
+ \lambda_2 (s_2 + 0) + \cdots + \lambda_r (s_r + 0) \\
+
(\mu_1/N) (s_1 + N t_1) + \cdots + (\mu_{r'}/N) (s_1 + N t_{r'}) \in \Cone_{\KK}(S+T)
\end{multline*}
because $0, N t_{1}, \ldots, N t_{r'} \in T$.
\end{proof}

\begin{Lemma}
\label{lem:conv:RR:QQ}
Let $P$ be a vector space over $\QQ$, $x_1, \ldots, x_r \in P$,
$b_1, \ldots, b_m \in \QQ$ and $A$ a $(r \times m)$-matrix whose entries belong to $\QQ$.
Let $\lambda_1, \ldots, \lambda_r \in \RR_{\geq 0}$ with
$(\lambda_1, \ldots, \lambda_r) A =
(b_1, \ldots, b_m)$.
If  $x : = \lambda_1 x_1 + \cdots + \lambda_r x_r \in P$, then
there are $\lambda'_1, \ldots, \lambda'_r \in \QQ_{\geq 0}$ such that
$x = \lambda'_1 x_1 + \cdots + \lambda'_r x_r$ and
$(\lambda'_1, \ldots, \lambda'_r) A =
(b_1, \ldots, b_m)$.
Moreover, if $\lambda_i$'s are positive, then we can choose positive $\lambda'_i$'s.
\end{Lemma}

\begin{proof}
If $\lambda_i = 0$, then 
\[
x = \sum_{j \not=i} \lambda_j x_j\quad\text{and}\quad
(\lambda_1, \ldots, \lambda_{i-1}, \lambda_{i+1}, \ldots, \lambda_r) A' =
(b_1, \ldots, b_m),
\]
where $A'$ is the $(r-1)\times n$-matrix obtained by deleting the $i$-th row from $A$.
Thus we may assume that $\lambda_i > 0$ for all $i$.
Let $e_1, \ldots, e_n$ be a basis of $\langle x_1, \ldots, x_r, x \rangle_{\QQ}$.
We set $x_i = \sum_{j} c_{ij} e_j$ and $x = \sum_{j} d_j e_j$ ($c_{ij} \in \QQ$, $d_j \in \QQ$).
Then $d_j = \sum_{i} \lambda_i c_{ij}$. Let $C = (c_{ij})$ and we consider linear maps
$f_{\QQ} : \QQ^r \to \QQ^{m+n}$ and $f_{\RR} : \RR^r \to \RR^{m+n}$ given by
\[
f_{\QQ}(s_1, \ldots, s_r) = (s_1, \ldots, s_r)(A, C)\quad\text{and}\quad
f_{\RR}(t_1, \ldots, t_r) = (t_1, \ldots, t_r)(A, C).
\]
Then $f_{\RR}(\lambda_1, \ldots, \lambda_r) = (b_1, \ldots, b_m, d_1, \ldots, d_n)$, that is,
\[
(b_1, \ldots, b_m, d_1, \ldots, d_n) \in f_{\RR}(\RR^r) \cap \QQ^{m+n}.
\]
Note that  $f_{\RR}(\RR^r) \cap \QQ^{m+n} = f_{\QQ}(\QQ^r)$ because
\[
\QQ^{m+n}/f_{\QQ}(\QQ^r) \to (\QQ^{m+n}/f_{\QQ}(\QQ^r)) \otimes_{\QQ} \RR
\]
is injective and
\[
 (\QQ^{m+n}/f_{\QQ}(\QQ^r)) \otimes_{\QQ} \RR
= (\QQ^{m+n} \otimes_{\QQ} \RR)/(f_{\QQ}(\QQ^r) \otimes_{\QQ} \RR) =
\RR^{m+n}/f_{\RR}(\RR^r).
\]
Therefore there is $(e_1, \ldots, e_r) \in \QQ^r$ with $f_{\QQ}(e_1, \ldots, e_r) = (b_1, \ldots, b_m, d_1, \ldots, d_n)$, and hence
\[
\begin{cases}
f_{\QQ}^{-1}(b_1, \ldots, b_m, d_1, \ldots, d_n) = f_{\QQ}^{-1}(0) + (e_1, \ldots, e_r), \\
f_{\RR}^{-1}(b_1, \ldots, b_m, d_1, \ldots, d_n) = f_{\RR}^{-1}(0) + (e_1, \ldots, e_r).
\end{cases}
\]
In particular, $f_{\QQ}^{-1}(b_1, \ldots, b_m, d_1, \ldots, d_n)$ is dense in $f_{\RR}^{-1}(b_1, \ldots, b_m, d_1, \ldots, d_n)$.
Thus, as  $(\lambda_1, \ldots, \lambda_r) \in f_{\RR}^{-1}(b_1, \ldots, b_m, d_1, \ldots, d_n) \cap \RR_{>0}^r$, we have
\[
f_{\QQ}^{-1}(b_1, \ldots, b_m, d_1, \ldots, d_n) \cap \RR^r_{>0} \not= \emptyset,
\]
that is,
we can find $(\lambda'_1, \ldots, \lambda'_r) \in \QQ_{>0}^r$ with $f_{\QQ}(\lambda'_1, \ldots, \lambda'_r) = (b_1, \ldots, b_m, d_1, \ldots, d_n)$.
Hence 
\[
x = \lambda'_1 x_1 + \cdots + \lambda'_r x_r \quad\text{and}\quad
(\lambda'_1, \ldots, \lambda'_r) A =
(b_1, \ldots, b_m).
\]
\end{proof}

Next we consider the following proposition.

\begin{Proposition}
\label{prop:A:V:open}
Let $P$ be a vector space over $\QQ$ and $V = P \otimes_{\QQ} \RR$.
Let $C$ be a cone in $P$.
Then we have the following:
\begin{enumerate}
\renewcommand{\labelenumi}{\rom{(\arabic{enumi})}}
\item
$\Cone_{\RR}(C) \cap P = C$.

\item
If $C$ is open, then $\Cone_{\RR}(C)$ is also open.

\item
If $D$ is a cone in $P$ with $0 \in D$, then
$\Cone_{\RR}(C+D) = \Cone_{\RR}(C) + \Cone_{\RR}(D)$.
\end{enumerate}
\end{Proposition}

\begin{proof}
(1) Clearly $C \subseteq \Cone_{\RR}(C) \cap P$.
We assume that $x \in \Cone_{\RR}(C) \cap P$.
Then, by Lemma~\ref{lem:conv:RR:QQ},
there are $\omega_1, \ldots, \omega_r \in C$ and $\lambda_1, \ldots, \lambda_r \in \QQ_{>0}$ with
$x = \lambda_1 \omega_1 + \cdots + \lambda_r \omega_r$,
which means that $x = \lambda_1 \omega_1 + \cdots + \lambda_r \omega_r \in C$.

\medskip
(2) First let us see the following:
for $a \in C$ and $x \in P$, there is $\delta_0 \in \QQ_{>0}$ such that
$a + \delta x \in \Cone_{\RR}(C)$ for all $\delta \in \RR$ with $\vert \delta \vert \leq \delta_0$.
Indeed, by our assumption, there is $\delta_0 \in \QQ_{>0}$ such that $a \pm \delta_0 x \in C$.
For $\delta \in \RR$ with $\vert \delta \vert \leq \delta_0$, if we set
$\lambda = (\delta + \delta_0)/2\delta_0$, then $0 \leq \lambda \leq 1$ and $\delta = \lambda\delta_0 + (1-\lambda)(-\delta_0)$.
Thus $b + \delta x = \lambda (b + \delta_0 x) + (1-\lambda)(b -\delta_0 x) \in \Cone_{\RR}(C)$.

By (1) in Proposition~\ref{prop:big:open}, it is sufficient to see that, for $a' \in \Cone_{\RR}(C)$ and $x' \in V$, there is a positive $\delta' \in \RR_{> 0}$
with $a' + \delta' x' \in \Cone_{\RR}(C)$.
We set $a' = \lambda_1 a_1 + \cdots + \lambda_r a_r$ ($a_1, \ldots, a_r \in C$, $\lambda_1, \ldots, \lambda_r \in \RR_{>0}$) and
$x' = \mu_1 x_1 + \cdots + \mu_n x_n$ ($x_1, \ldots, x_n \in P$, $\mu_1, \ldots, \mu_n \in \RR$).
We choose $\lambda \in \QQ$ such that $0 < \lambda < \lambda_{1}$.
By the above claim,
there is  $\delta_0 \in \QQ_{>0}$ such that
$(\lambda/n)a_1 + \delta x_j \in \Cone_{\RR}(C)$ for all $j$ and all $\delta \in \RR$ with $\vert \delta \vert \leq \delta_0$. 
We choose $\delta' \in \RR_{>0}$ such that $\vert \delta' \mu_j \vert \leq \delta_0$ for all $j$.
Then
\[
a' + \delta' x = (\lambda_1 - \lambda)a_1 + \sum_{i\geq 2} \lambda_i a_i + \sum_{j=1}^n ((\lambda/n) a_1 + \delta' \mu_j x_j) \in \Cone_{\RR}(C),
\]
as required.

(3) follows from Lemma~\ref{lem:cone:sum}.
\end{proof}

Let $M$ be a $\ZZ$-module and $A$ a sub-semigroup of $M$.
$A$ is said to be {\em open} if,
for any $a \in A$ and $x \in M$, there is a positive integer $n$ such that $n a + x \in A$.
For example, let $X$ be a projective arithmetic variety and $\aAmp(X)$ the sub-semigroup of $\aPic(X;C^{\infty})$
consisting of ample $C^{\infty}$-hermitian invertible sheaves on $X$.
Then $\aAmp(X)$ is open as a sub-semigroup of $\aPic(X;C^{\infty})$ (cf. Conventions and terminology~\ref{CT:Arakelov:positivity}).

\begin{Proposition}
\label{prop:QSat}
Let $\iota : M \to M \otimes_{\ZZ} \QQ$ be the natural homomorphism, and
let $A$ be sub-semigroups of $M$.
Then we have the following.
\begin{enumerate}
\renewcommand{\labelenumi}{\rom{(\arabic{enumi})}}
\item
$\Cone_{\QQ}(\iota(A)) = \{ (1/n)\iota(a) \mid n \in \ZZ_{>0}, \ a \in A \}$.

\item $\QSat(A) = \iota^{-1}(\Cone_{\QQ}(\iota(A))$ 
\rom{(}see Conventions and terminology~\rom{\ref{CT:sub:semigroup:monoid}} for the saturation $\QSat(A)$ of $A$ in $M$\rom{)}.

\item
If $A$ is open, then $\Cone_{\QQ}(\iota(A))$ is an open set in $M \otimes_{\ZZ} \QQ$.

\item
If $B$ is a sub-monoid of $M$, then
\[
\Cone_{\QQ}(\iota(A + B)) = \Cone_{\QQ}(\iota(A)) + \Cone_{\QQ}(\iota(B)).
\]

\item
Let $f : A \to \RR$ be a function on $A$. If  there is a positive real number $e$ such that
$f(n a) = n^e f(a)$ for all $n \in \ZZ_{>0}$ and $a \in A$, then there is a unique function
$\tilde{f} : \Cone_{\QQ}(\iota(A)) \to \RR$ with the following properties:
\begin{enumerate}
\renewcommand{\labelenumii}{\rom{(\arabic{enumi}.\arabic{enumii})}}
\item
$\tilde{f} \circ \iota = f$.

\item
$\tilde{f}(\lambda x ) = \lambda^e f(x)$ for all $\lambda \in \QQ_{>0}$ and $x \in \Cone_{\QQ}(\iota(A))$.
\end{enumerate}
\end{enumerate}
\end{Proposition}

\begin{proof}
(1) Let $x \in \Cone_{\QQ}(\iota(A))$. Then there are positive integers $n, m_1, \ldots, m_r$ and
$a_1, \ldots, a_r \in A$ such that $x = (m_1/n) \iota(a_1) + \cdots + (m_r/n) \iota(a_r)$.
Thus, if we set $a = m_1 a_1 + \cdots + m_r a_r \in A$,
then $x = (1/n) \iota(a)$. The converse is obvious.

(2) Clearly $\iota^{-1}(\Cone_{\QQ}(\iota(A))$ is saturated, and hence 
\[
\QSat(A) \subseteq \iota^{-1}(\Cone_{\QQ}(\iota(A)).
\]
Conversely we assume that $x \in \iota^{-1}(\Cone_{\QQ}(\iota(A))$.
Then, by (1), there are $n \in \ZZ_{>0}$ and $a \in A$ such that $\iota(x) = (1/n) \iota(a)$.
Thus, as $\iota(nx - a) = 0$, 
there is $n' \in \ZZ_{>0}$ such that $n'(n x - a) = 0$,
which means that $n'n x \in A$, as required.

(3) By (1) in Proposition~\ref{prop:big:open},
it is sufficient to show that, for any $a' \in \Cone_{\QQ}(\iota(A))$ and $x' \in M \otimes \QQ$,
there is $\delta \in \QQ_{>0}$ such that $a' + \delta x' \in \Cone_{\QQ}(\iota(A))$.
We can choose $a \in A$, $x \in M$ and positive integers $n_1$ and $n_2$ such that $a' = (1/n_1)\iota(a)$ and
$x' = (1/n_2)\iota(x)$.
By our assumption, there is a positive integer $n$ such that $n a + x \in A$.
Thus 
\[
n n_1 a' + n_2 x' = \iota(na + x) \in \iota(A),
\]
which yields $a' + (n_2/nn_1) x' \in \Cone_{\QQ}(\iota(A))$.

(4)
By virtue of Lemma~\ref{lem:cone:sum},
\[
\Cone_{\RR}(\iota(A + B)) =  \Cone_{\RR}(\iota(A)+ \iota(B)) = \Cone_{\RR}(\iota(A))+ \Cone_{\RR}(\iota(B)).
\]

(5)
First let us see the uniqueness of $\tilde{f}$.
Indeed, if it exists, then
\[
\tilde{f}((1/n) \iota(a)) = (1/n)^e \tilde{f}(\iota(a)) = (1/n)^e f(a).
\]
By the above observation,
in order to define $\tilde{f} :  \Conv_{\QQ}(\iota(A)) \to \RR$,
it is sufficient to show that if $(1/n) \iota(a) = (1/n')\iota(a')$ ($n, n' \in \ZZ_{>0}$ and $a, a' \in A$), then
$(1/n)^e f(a) = (1/n')^e f(a')$.
As $\iota(n'a - na') = 0$, there is $m \in \ZZ_{>0}$ such that $mn'a = mna'$.
Thus
\[
(mn')^e f(a) = f((mn')a) = f((mn)a') = (mn)^e f(a),
\]
which implies that $(1/n)^e f(a) = (1/n')^e f(a')$.
Finally let us see (5.2). We choose positive integers $n, n_1, n_2$ and $a \in A$ such that
$\lambda = n_1/n_2$ and $x = (1/n) \iota(a)$.
Then
\begin{align*}
\tilde{f}(\lambda x) & = \tilde{f}((1/n_2n)\iota(n_1a)) = (1/n_2n)^e f(n_1a) = (1/n_2n)^e n_1 ^e f(a) \\
& = \lambda^e (1/n)^e f(a) =
\lambda^e \tilde{f}(x).
\end{align*}
\end{proof}

\setcounter{Theorem}{0}
\subsection{Convex lattice}
\label{subsec:convex:lattice}
Let $M$ be a finitely generated free $\ZZ$-module.
Let $K$ be a subset of $M$.
The $\ZZ$-submodule generated by $K$ in $M$ and
the convex hull of $K$ in $M_{\RR} := M \otimes_{\ZZ} \RR$ are denoted by
$\langle K \rangle_{\ZZ}$ and $\Conv_{\RR}(K)$ respectively.
For a positive integer $m$,
the $m$-fold sum $m \ast K$ of elements in $K$ is defined to be
\[
m \ast K = \{ x_1 + \cdots + x_m \mid x_1, \ldots, x_m \in K \}.
\]
We say $K$ is a {\em convex lattice} if 
\[
 \langle K \rangle_{\ZZ} \cap \frac{1}{m}(m \ast K)
\subseteq K,\quad\text{that is}, \quad
m \langle K \rangle_{\ZZ} \cap (m \ast K)
\subseteq m K
\]
holds for all $m \geq 1$.
Moreover,  $K$ is said to be {\em symmetric} if $-x \in K$ for all $x \in K$.
Note that if $K$ is symmetric, then $\Conv_{\RR}(K)$ is also symmetric.

\begin{Proposition}
Let $K$ be a subset of $M$. Then we have the following:
\begin{enumerate}
\renewcommand{\labelenumi}{\rom{(\arabic{enumi})}}
\item
${\displaystyle  \langle K \rangle_{\ZZ} \cap \bigcup_{m=1}^{\infty} \frac{1}{m}(m \ast K) = \langle K \rangle_{\ZZ} \cap \Conv_{\RR}(K)}$.

\item
The following are equivalent:
\begin{enumerate}
\renewcommand{\labelenumii}{\rom{(\arabic{enumi}.\arabic{enumii})}}
\item $K$ is a convex lattice.

\item $K =  \langle K \rangle_{\ZZ} \cap \Conv_{\RR}(K)$.

\item There are a $\ZZ$-submodule $N$ of $M$ and a convex set $\Delta$ in $M_{\RR}$ such that
$K = N \cap \Delta$.
\end{enumerate}
\end{enumerate}
\end{Proposition}

\begin{proof}
(1) Obviously $ \langle K \rangle_{\ZZ}  \cap \bigcup_{m=1}^{\infty} (1/m)(m \ast K) \subseteq \langle K \rangle_{\ZZ} \cap \Conv_{\RR}(K)$.
We assume that $x \in \langle K \rangle_{\ZZ} \cap \Conv_{\RR}(K)$. Then
there are $a_1, \ldots, a_l \in K$ and $\mu_1, \ldots, \mu_l \in \RR_{\geq 0}$ such that
$x = \mu_1 a_1 + \cdots + \mu_l a_l$ and $\mu_1 + \cdots + \mu_l = 1$.
As $x \in \langle K \rangle_{\ZZ} \subseteq M$,
by using Lemma~\ref{lem:conv:RR:QQ}, 
we can find $\lambda_1, \ldots, \lambda_l \in \QQ_{\geq 0}$ such that
$\lambda_1 + \cdots + \lambda_l = 1$ and
$x = \lambda_1 a_1 + \cdots + \lambda_l a_l$.
We set $\lambda_i = d_i/m$ for $i=1, \ldots, l$.
Then, as  $d_1 + \cdots + d_l = m$, we have
\[
x = \frac{d_1 x_1 + \cdots + d_l x_l}{m} \in \langle K \rangle_{\ZZ} \cap \frac{1}{m}(m \ast K).
\]

\medskip
(2)
(2.1) $\Longrightarrow$ (2.2) : Since $K$ is a convex lattice, by (1),
$K =  \langle K \rangle_{\ZZ} \cap \Conv_{\RR}(K)$.

(2.2) $\Longrightarrow$ (2.3) is obvious.

(2.3) $\Longrightarrow$ (2.1) : 
First of all, note that $\langle K \rangle_{\ZZ} \subseteq N$ and
$\Conv_{\RR}(K) \subseteq \Delta$. 
Thus
\[
\langle K \rangle_{\ZZ}  \cap \frac{1}{m} (m \ast K)  \subseteq  \langle K \rangle_{\ZZ} \cap \Conv_{\RR}(K) \subseteq
N \cap \Delta  = K.
\]
\end{proof}

Let $K$ be a subset of $M$.
Then, by the above proposition, 
\[
\langle K \rangle_{\ZZ} \cap \bigcup_{m=1}^{\infty} \frac{1}{m}(m \ast K) =
\{ x \in \langle K \rangle_{\ZZ} \mid \exists m \in \ZZ_{> 0}\ \ mx \in m \ast K \}
\]
is a convex lattice, so that
it is called the {\em convex lattice hull} of $K$ and is denoted by $\CL(K)$.
Note that the convex lattice hull of $K$ is the smallest convex lattice containing $K$.
Let $f : M \to M'$ be an injective homomorphism of finitely generated free $\ZZ$-modules.
Then it is easy to see that $f(\CL(K)) = \CL(f(K))$.
Finally we consider the following lemma.
Ideas for the proof of the lemma can be found in Yuan's paper \cite[\S2.3]{YuanVol}.

\setcounter{equation}{0}
\begin{Lemma}
\label{lem:conv:module:rest:estimate}
Let $M$ be a finitely generated free $\ZZ$-module and
$r : M \to N$ a homomorphism of finitely generated $\ZZ$-modules.
For a symmetric finite subset  $K$ of $M$, we have the following estimation:
\addtocounter{Claim}{1}
\begin{equation}
\label{eqn:lem:conv:module:rest:estimate:1}
\log \# r(K) \geq
\log \#(K) - \log \#( \Ker(r) \cap (2 \ast K)).
\end{equation}
\addtocounter{Claim}{1}
\begin{equation}
\label{eqn:lem:conv:module:rest:estimate:2}
\log\# r(K) \leq \log \# (2 \ast K) 
- \log \#(\Ker(r) \cap K).
\end{equation}
Moreover, if $\Delta$ is a bounded and symmetric convex set in $M_{\RR}$ and
$a$ is a real number with $a \geq 1$, then
\addtocounter{Claim}{1}
\begin{equation}
\label{eqn:lem:conv:module:rest:estimate:3}
0 \leq \log \#(M \cap a \Delta) - \log \#(M \cap \Delta)
\leq
\log(\lceil 2a \rceil) \rank M.
\end{equation}
\end{Lemma}

\begin{proof}
Let $t \in r(K)$ and fix $s_0 \in K$ with
$r(s_0) = t$.
Then, for any $s \in r^{-1}(t) \cap K$,
\[
s - s_0 = s + (-s_0) \in \Ker(r) \cap (2 \ast K).
\]
Thus
\[
\#(r^{-1}(t) \cap K) \leq \#( \Ker(r) \cap (2 \ast K)).
\]
Therefore,
\[
\#(K) = \sum_{t \in r(K)} \#(r^{-1}(t) \cap K)
\leq \#(r(K)) \#( \Ker(r) \cap (2 \ast K)),
\]
as required.

\medskip
We set $S = K +  \Ker(r) \cap K$.
Then $r(S) = r(K)$ and
$S \subseteq 2 \ast K$. Moreover, for all $t \in r(S)$,
\[
\#(\Ker(r) \cap K) \leq \#(S \cap r^{-1}(t)).
\]
Indeed, if we choose $s_0 \in K$ with $r(s_0) = t$,
then 
\[
s_0 +  \Ker(r) \cap K \subseteq S \cap r^{-1}(t).
\]
Therefore,
\begin{align*}
\#(2 \ast K) & \geq \#(S) = \sum_{t \in r(S)} \#(r^{-1}(t) \cap S) \geq
\#(r(S)) \#(\Ker(r) \cap K) \\
& = \#(r(K))  \#(\Ker(r) \cap K)
\end{align*}
as required.

\medskip
We set $n =  \lceil 2a \rceil$.
Applying \eqref{eqn:lem:conv:module:rest:estimate:1} to the case where
$K = M \cap (n/2)\Delta$ and $r : M \to M/nM$, we have
\begin{multline*}
\log \#(M \cap (n/2) \Delta) - \log \#(nM \cap 2 \ast ((n/2) \Delta \cap M)) \\
\leq
\log \# M/nM = \log(n) \rank M.
\end{multline*}
Note that $a \leq n/2$ and
\begin{multline*}
\#(nM \cap 2 \ast ((n/2) \Delta \cap M)) \leq \#(nM \cap (n\Delta \cap M)) \\
= \#(nM \cap n\Delta) = \#(M \cap \Delta).
\end{multline*}
Hence we obtain
\begin{multline*}
0 \leq \log \#(M \cap a \Delta) - \log \#(M \cap \Delta) \\
\leq
\log \#(M \cap (n/2)\Delta) - \log \#(nM \cap 2 \ast ((n/2) \Delta \cap M))
\leq
\log(n) \rank M.
\end{multline*}
\end{proof}

\setcounter{Theorem}{0}
\subsection{Concave function and its continuity}
Let $P$ be a vector space over $\QQ$ and $V = P \otimes \RR$.
Let $C$ be a non-empty  open convex set in $V$.
Let  $f : C \cap P \to \RR$ be a concave function over $\QQ$
(cf. Conventions and terminology~\ref{CT:convex:set}).

We assume that $P$ is finite dimensional and $d = \dim_{\QQ} P$.
Let $h$ be an inner product of $V$.
For $x \in V$, we denote $\sqrt{h(x,x)}$ by $\Vert x \Vert_h$.
Moreover, for a positive number $r$ and $x \in V$,
we set 
\[
U(x, r) = \{ y \in V \mid \Vert y - x \Vert_h < r \}.
\]

\begin{Proposition}
\label{prop:concave:Lipschitz:cont}
For any $x \in C$, there are positive numbers  $\epsilon$ and $L$ such that
$U(x, \epsilon) \subseteq C$ and $\vert f(y) - f(z) \vert \leq L\Vert y - z \Vert_h$ for all
$y, z \in U(x, \epsilon) \cap P$.
In particular, there is a unique concave and continuous function $\tilde{f} : C \to \RR$
such that $\rest{\tilde{f}}{C \cap P} = f$.
\end{Proposition}

\begin{proof}
The proof of this proposition is almost same as one of  \cite[Theorem 2.2]{Gru},
but we need a slight modification because  $x$ is not necessarily a point of $P$.
Let us begin with the following claim.

\begin{Claim}
\label{claim:prop:concave:Lipschitz:cont:1}
$f(t_1 x_1 + \cdots + t_r x_r) \geq t_1 f(x_1) + \cdots + t_r f(x_r)$ holds for any $x_1, \ldots, x_r \in C \cap P$ and
$t_1, \ldots, t_r \in \QQ_{\geq 0}$ with $t_1 + \cdots + t_r = 1$.
\end{Claim}

We prove it by induction on $r$. In the case where $r =1, 2$, the assertion is obvious.
We assume $r \geq 3$. If $t_1 = 1$, then the assertion is also obvious, so that we may assume that $t_1 < 1$.
Then, by using the hypothesis of induction,
\begin{align*}
f(t_1 x_1 + \cdots + t_r x_r) & = f\left( t_1 x_1 + (1-t_1) \left( \frac{t_2}{1-t_1} x_2 + \cdots + \frac{t_r}{1-t_1} x_r \right)\right) \\
& \geq t_1 f(x_1) + (1-t_1) f\left( \frac{t_2}{1-t_1} x_2 + \cdots + \frac{t_r}{1-t_1} x_r \right) \\
& \geq t_1 f(x_1) + (1-t_1) \left( \frac{t_2}{1-t_1}f(  x_2)  + \cdots + \frac{t_r}{1-t_1} f(x_r) \right) \\
& =  t_1 f(x_1) + \cdots + t_r f(x_r).
\end{align*}
\CQED

\begin{Claim}
\label{claim:prop:concave:Lipschitz:cont:2}
There are $x_1, \ldots, x_{d+1} \in C \cap P$ such that
$x$ is an interior point of $\Conv_{\RR}(\{ x_1, \ldots, x_{d+1} \})$.
\end{Claim}

Let us consider the function $\phi : C^d \to \RR$ given by
\[
\phi(y_1, \ldots, y_d) = \det (y_1 - x, \ldots, y_d - x).
\]
Then $(C^d)_{\phi} = \{ (y_1, \ldots, y_d) \in C^d\mid \phi(y_1, \ldots, y_d) \not= 0\}$ is a non-empty open set, so that
we can find $(x_1, \ldots, x_d) \in (C^d)_{\phi}$ with $x_1, \ldots, x_d \in P$.
Next we consider 
\[
\{ x - t_1 (x_1 - x) - \cdots - t_d (x_d - x) \mid t_1, \ldots, t_d \in \RR_{> 0} \} \cap C.
\]
This is also a non-empty open set in $C$.
Thus there are $x_{d+1} \in C \cap P$ and $t_1, \ldots, t_d \in \RR_{>0}$
with $x_{d+1} = x - t_1 (x_1 - x) - \cdots - t_d (x_d - x)$, so that
\[
x = \frac{t_1 x_1 + \cdots + t_d x_d + x_{d+1}}{t_1 + \cdots + t_d + 1}.
\]
Thus $x$ is an interior point of $\Conv_{\RR}(\{ x_1, \ldots, x_{d+1} \})$.
\CQED

\begin{Claim}
\label{claim:prop:concave:Lipschitz:cont:3}
There is a positive number $c_1$ such that
$f(y) \geq - c_1$ holds for all $y \in \Conv_{\RR}(\{ x_1, \ldots, x_{d+1} \}) \cap P$.
\end{Claim}

As $y \in \Conv_{\RR}(\{ x_1, \ldots, x_{d+1} \}) \cap P$, by Lemma~\ref{lem:conv:RR:QQ}, there are $t_1, \ldots, t_{d+1} \in \QQ_{\geq 0}$ such that
\[
t_1 + \cdots  + t_{d+1} = 1\quad\text{and}\quad
y = t_1 x_1 + \cdots + t_{d+1} x_{d+1}.
\]
Thus, by Claim~\ref{claim:prop:concave:Lipschitz:cont:1},
\begin{multline*}
f(y) = f(t_1 x_1 + \cdots + t_{d+1} x_{d+1}) \\
\geq
t_1 f(x_1) + \cdots + t_{d+1} f(x_{d+1}) \geq -t_1 \vert f(x_1)\vert  - \cdots - t_{d+1} \vert f(x_{d+1}) \vert \\
\geq -(\vert f(x_1)\vert  + \cdots +  \vert f(x_{d+1}) \vert),
\end{multline*}
as required.
\CQED

\medskip
Let us choose a positive number $\epsilon$ and $x_0 \in P$ such that $U(x, 4\epsilon) \subseteq \Conv_{\RR}(\{ x_1, \ldots, x_{d+1} \})$ and
$x_0 \in U(x, \epsilon) \cap P$. Then 
\[
U(x, \epsilon) \subseteq U(x_0, 2\epsilon) \subseteq U(x_0, 3 \epsilon) \subseteq U(x, 4 \epsilon) \subseteq \Conv_{\RR}(\{ x_1, \ldots, x_{d+1} \}).
\]

\begin{Claim}
\label{claim:prop:concave:Lipschitz:cont:4}
There is a positive number $c_2$ such that $\vert f(y) \vert \leq c_2$ holds for all $y \in U(x_0, 3\epsilon) \cap P$.
\end{Claim}

As $(2x_0 - y) - x_0 = x_0 - y$,  we have $2x_0 - y \in U(x_0, 3\epsilon) \cap P$, and hence
\[
f(x_0) = f(y/2 + (2x_0 - y)/2) \geq f(y)/2 + f(2x_0 - y)/2.
\] 
Therefore,
\[
-c_1 \leq f(y) \leq 2 f(x_0) - f(2x_0 - y) \leq 2f(x_0) + c_1,
\]
as required.
\CQED

\medskip
Let $y, z \in U(x, \epsilon) \cap P$ with $y \not= z$.
We choose $a \in \QQ$ with
\[
\frac{\epsilon/2}{\Vert z - y \Vert_h} + 1 \leq a \leq \frac{\epsilon}{\Vert z - y \Vert_h} + 1
\]
and we set $w = a(z-y) + y$. Then
$\epsilon/2 \leq \Vert w - z \Vert_h \leq \epsilon$. Thus $w \in U(x_0, 3\epsilon) \cap P$.
Moreover, if we put $t_0 = 1/a$, then 
\[
z = (1-t_0)y + t_0w,\quad
z-y = t_0(w-y)\quad\text{and}\quad
w-z = (1-t_0)(w- y).
\]
As $\Vert z - y \Vert_h/\Vert w - z \Vert_h = t_0/(1-t_0)$, we have
\begin{align*}
\frac{f(z) - f(y) }{\Vert z - y \Vert_h} & = \frac{ f((1-t_0)y + t_0w) - f(y)}{\Vert z - y \Vert_h} \geq \frac{ (1-t_0)f(y)  + t_0f(w) - f(y)}{\Vert z - y \Vert_h} \\
& = t_0 \frac{ f(w) - f(y)}{\Vert z - y \Vert_h} = (1-t_0) \frac{ f(w) - f(y)}{\Vert w-z\Vert_h} \\
& = \frac{ f(w) - \left( (1-t_0)f(y) + t_0 f(w)\right)}{\Vert w-z\Vert_h} \geq \frac{f(w) - f(z)}{\Vert w - z \Vert_h} \\
& \geq \frac{-2c_2}{\Vert w - z \Vert_h} \geq 
\frac{-2c_2}{\epsilon/2} = \frac{-4c_2}{\epsilon}.
\end{align*}
Exchanging $y$ and $z$, we obtain the same inequality, that is,
\[
\frac{f(y) - f(z) }{\Vert y - z \Vert_h} \ \geq  \frac{-4c_2}{\epsilon}.
\]
Therefore, $\vert f(z) - f(y) \vert \leq (4c_2/\epsilon) \Vert y - z \Vert_h$ for all $y, z \in U(x, \epsilon) \cap P$.

For the last assertion, note the following:
Let $\{ a_n \}_{n=1}^{\infty}$ be a Cauchy sequence on $C\cap P$ such that
$x = \lim_{n\to\infty} a_n \in C$.
Then, by the first assertion of this proposition, $\{ f(a_n) \}_{n=1}^{\infty}$ is also a Cauchy sequence in $\RR$, and hence
$\tilde{f}(x)$ is defined by $\lim_{n\to\infty} f(a_n)$.
\end{proof}

Next we do not assume that $P$ is finite dimensional. Then we have the following corollary.

\begin{Corollary}
\label{cor:concave:extension}
There is a unique concave and continuous function $\tilde{f} : C \to \RR$
such that $\rest{\tilde{f}}{C \cap P} = f$.
\end{Corollary}

\begin{proof}
It follows from Proposition~\ref{prop:concave:Lipschitz:cont} and the following facts:
If $x \in V$, then there is a finite dimensional vector subspace $Q$ of $P$ over $\QQ$ with $x \in Q \otimes \RR$.
\end{proof}

\setcounter{Theorem}{0}
\subsection{Good flag over a prime}
\label{subsection:good:flag:over:prime}
In this subsection, we observe the existence of good flags over
infinitely many prime numbers.

Let $X$ be a $d$-dimensional projective arithmetic variety.
Let $\pi : X \to \Spec(R)$ be the Stein factorization of
$X \to \Spec(\ZZ)$, where $R$ is an order of some number field $F$.
A chain 
\[
Y_{\cdot} : Y_0 = X \supset Y_1 \supset Y_2 \supset \cdots \supset Y_d
\]
of subschemes of $X$
is called a {\em good flag} of $X$ over a prime $p$ if the following conditions are satisfied:
\begin{enumerate}
\renewcommand{\labelenumi}{\rom{(\alph{enumi})}}
\item $Y_i$'s are integral and $\codim(Y_i) = i$ for $i=0, \ldots, d$.

\item There is $P \in \Spec(R)$ such that $R_P$ is normal,
$\pi^{-1}(P) = Y_1$ and the residue field $\kappa(P)$ at $P$ is isomorphic to $\FF_p$.
In particular, $Y_1$ is a Cartier divisor on $X$.

\item $Y_d$ consists of  a rational point $y$ over $\FF_p$.

\item $Y_i$'s are regular at $y$ for $i=0, \ldots, d$.

\item
There is a birational morphism $\mu : X' \to X$ of projective arithmetic varieties with the following properties:
\begin{enumerate}
\renewcommand{\labelenumii}{\rom{(\alph{enumi}.\arabic{enumii})}}
\item
$\mu$ is isomorphism over $y$.

\item
Let $Y'_i$ be the strict transform of $Y_i$.
Then $Y'_{i}$ is a Cartier divisor in $Y'_{i-1}$ for $i=1, \ldots, d$.
\end{enumerate}
\end{enumerate}

\setcounter{equation}{0}
\begin{Proposition}
There are good flags of $X$ over infinitely many prime numbers.
More precisely, if we set
$\mathcal{S}_{F/\QQ} = \{ p \in \Spec(\ZZ) \mid \text{$p$ splits completely in $F$ over $\QQ$}\}$,
then there is a finite subset $\Sigma$ of $\mathcal{S}_{F/\QQ}$ such that
we have a good flag over a prime in $\mathcal{S}_{F/\QQ} \setminus \Sigma$.
\end{Proposition}

\begin{proof}
Let $\mu : Y \to X$ be a generic resolution of singularities of $X$ such that
$Y$ is normal.
Let $\pi : X \to \Spec(R)$ and $\tilde{\pi} : Y \to \Spec(O_F)$ be the Stein factorizations
of $X \to \Spec(\ZZ)$ and $Y \to \Spec(\ZZ)$ respectively.
Then we have the following commutative diagram:
\[
\begin{CD}
X @<{\mu}<< Y \\
@V{\pi}VV @VV{\tilde{\pi}}V \\
\Spec(R) @<{\rho}<< \Spec(O_F).
\end{CD}
\]
Let us choose a proper closed subset $Z$ of $X$ such that
$\mu : Y \setminus \mu^{-1}(Z) \to X \setminus Z$ is an isomorphism.
We set $E = \mu^{-1}(Z)$.
Let us choose a chain 
\[
Y'_1 = Y \times_{\Spec(O_F)} \Spec(F) \supset Y'_2 \supset \cdots \supset Y'_{d-1}
\]
of smooth subvarieties of $Y \times_{\Spec(O_F)} \Spec(F)$ such that
$\codim(Y'_i) = i-1$ for $i=1, \ldots, d-1$ and
$\dim (Y'_{d-1} \cap (E \times_{\Spec(O_F)} \Spec(F))) \leq 0$.
Let $Y_i$ be the Zariski closure of $Y'_i$ in $Y$.
Then there is a non-empty open set $U$ of $\Spec(O_F)$ such that
$\rho$ is an isomorphism over $U$,
$Y_1=Y, Y_2, \ldots, Y_{d-1}$ are smooth over $U$ and
$Y_{d-1} \cap E$ is either finite or empty over $U$.
Let $e$ be the degree of $Y_{d-1} \cap E$ over $U$.
Note that $e$ might be zero.
If we put
\[
\Sigma_1 = 
\{ p \in \Spec(\ZZ) \mid \text{there is $P \in \Spec(O_F) \setminus U$ with $p\ZZ = P \cap \ZZ$}\},
\]
then $\Sigma_1$ is a finite set.
Let $p \in \mathcal{S}_{F/\QQ} \setminus \Sigma_1$ and $P \in \Spec(O_K)$ with
$p\ZZ = P \cap \ZZ$. Then $P \in U$ and the residue field at $P$ is isomorphic to $\FF_p$.
By Weil's conjecture for curves,
\[
p + 1 - 2g\sqrt{p} \leq \#(Y_{d-1} \otimes \kappa(P))(\FF_{p}),
\]
where $g$ is the genus of $Y'_{d-1}$.
Thus there is a finite set $\Sigma_2$ such that,
if $p \in \mathcal{S}_{F/\QQ} \setminus (\Sigma_1 \cup \Sigma_2)$, then
$p + 1 - 2g \sqrt{p} > e$,
which means that there is
$x \in (Y_{d-1} \otimes \kappa(P))(\FF_{p})$ with $x \not\in E$.
Since, for $p \in\mathcal{S}_{F/\QQ} \setminus (\Sigma_1 \cup \Sigma_2)$,
\[
Y \supset Y_{1} \otimes \kappa(P) \supset \cdots \supset
Y_{d-1} \otimes \kappa(P) \supset \{ x \}
\]
is a good flag over $p$,
\[
X \supset \mu(Y_{1} \otimes \kappa(P)) \supset \cdots \supset
\mu(Y_{d-1} \otimes \kappa(P)) \supset \{ \mu(x) \}
\]
is also a good flag over $p$.
\end{proof}

\renewcommand{\theTheorem}{\arabic{section}.\arabic{Theorem}}
\renewcommand{\theClaim}{\arabic{section}.\arabic{Theorem}.\arabic{Claim}}
\renewcommand{\theequation}{\arabic{section}.\arabic{Theorem}.\arabic{Claim}}

\section{Estimation of linear series in terms of valuation vectors}

The context of this section is a generalization of Yuan's paper \cite{YuanVol}.
Let us begin with the following proposition, which is 
a key to Theorem~\ref{thm:estimate:valuation:vec}.

\setcounter{equation}{0}
\begin{Proposition}
\label{prop:estimate:val:on:Y:1}
Let $X$ be a $d$-dimensional projective arithmetic variety and
fix a good flag $Y_{\cdot} : X \supset Y_1 \supset Y_2 \supset \cdots \supset Y_d$ over a prime $p$.
Let $L$ be an invertible sheaf on $X$, $M$ a $\ZZ$-submodule of $H^0(X, L)$ and
$\Delta$ a bounded symmetric convex set in $H^0(X, L)_{\RR}$.
Let
$r : H^0(X, L) \to H^0(Y_1, \rest{L}{Y_1})$ be the natural homomorphism,
$M' = M \cap H^0(X, L - Y_1)$ and $\beta = p \rank M$.
Then we have the following:
\addtocounter{Claim}{1}
\begin{equation}
\label{eqn:prop:estimate:val:on:Y:1:1}
\# \nu_{Y_1}\left(r(M \cap \Delta) \setminus \{ 0 \} \right) \log p
\leq \log \#(M \cap 2\beta\Delta) - \log \#(M' \cap \beta \Delta)
\end{equation}
and
\addtocounter{Claim}{1}
\begin{equation}
\label{eqn:prop:estimate:val:on:Y:1:2}
\# \nu_{Y_1}\left(r(M \cap \Delta) \setminus \{ 0 \}\right) \log p
\geq \log \#(M \cap (1/\beta)\Delta) - \log \#(M' \cap (2/\beta) \Delta),
\end{equation}
where
$\nu_{Y_1}$ is the valuation on $Y_1$ attached to a flag
$Y_1 \supset Y_2 \supset \cdots \supset Y_d$.
\end{Proposition}

\begin{proof}
Let $V$ be a vector space generated by
$r(M \cap \Delta)$ in $H^0(Y_1, \rest{L}{Y_1})$ over $\FF_p$.
Note that \cite[Lemma~1.3]{LazMus} holds if $Y_d$ consists of a rational point over a base field.
Thus
\begin{align*}
\# \nu_{Y_1}\left(r(\Delta \cap M)\setminus \{ 0 \}\right)  \log p & \leq
\# \nu_{Y_1}\left( V \setminus \{ 0 \}\right)  \log p \\
& = \dim_{\FF_p} (V) \log p &  
(\because \text{\cite[Lemma~1.3]{LazMus}})
\\
& = \log \# V.
\end{align*}
Let us choose $s_1, \ldots, s_l \in M \cap \Delta$ such that
$r(s_1), \ldots, r(s_l)$ forms a basis of 
$V$.
Let $n$ be the rank of $M$ and $\omega_1, \ldots, \omega_n$ a free basis of $M$.
Then $V \subseteq \sum_{i=1}^n \FF_p r(\omega_i)$ in $H^0(Y_1, \rest{L}{Y_1})$, which implies
$l \leq n$.
We set
\[
S = \left\{ \sum a_i s_i \mid a_i = 0, 1, \ldots, p-1\ (\forall i) \right\}.
\]
Then $S$ maps surjectively to $V$.
Moreover $S \subseteq  M \cap \beta \Delta$ because $l \leq n$.
Thus we get $\# V \leq
\#\left( r(M \cap \beta \Delta)\right)$.
Note that $\Ker \left(\rest{r}{M} : M \to H^0(Y_1, \rest{L}{Y_1})\right) = M'$.
Therefore, as $2 \ast (M \cap \beta \Delta) \subseteq M \cap 2\beta \Delta$,
by \eqref{eqn:lem:conv:module:rest:estimate:2},
\[
\log \# r(M \cap \beta \Delta)  \leq
\log \#(M \cap 2\beta \Delta) - \log \#( M' \cap \beta \Delta),
\]
which shows \eqref{eqn:prop:estimate:val:on:Y:1:1}.

Let $W$ be a vector space generated by
$r(M \cap (1/\beta)\Delta)$ in $H^0(Y_1, \rest{L}{Y_1})$ over $\FF_p$.
Let us choose $t_1, \ldots, t_{l'} \in M \cap (1/\beta)\Delta$ such that
$r(t_1), \ldots, r(t_{l'})$ forms a basis of
$W$. In the same way as before, we have $l' \leq n$.
We set 
\[
T = \left\{ \sum b_i t_i \mid b_i = 0, 1, \ldots, p-1\ (\forall i) \right\}.
\]
Then $T \subseteq M \cap \Delta$ and
$W = r(T) \subseteq
r(M \cap \Delta)$.
Thus
\begin{align*}
\# \nu_{Y_1}\left(r(M \cap \Delta)\setminus \{ 0 \}\right)\log p
& \geq \# \nu_{Y_1}\left( W \setminus \{ 0 \}\right)\log p \\
& = \dim_{\FF_p} (W) \log p \\
& = \log \# W\\
& \geq  \log \# r\left(M \cap (1/\beta)\Delta\right).
\end{align*}
On the other hand, as $2 \ast (M \cap (1/\beta)\Delta) \subseteq M \cap (2/\beta) \Delta$,
by \eqref{eqn:lem:conv:module:rest:estimate:1},
\[
\log \# r(M \cap (1/\beta)\Delta)
\geq
\log \# (M \cap (1/\beta)\Delta) - \log \#(M' \cap (2/\beta)\Delta),
\]
as required for \eqref{eqn:prop:estimate:val:on:Y:1:2}.
\end{proof}

Let $X$ be a $d$-dimensional projective arithmetic variety and $\overline{L}$ a continuous
hermitian invertible sheaf on $X$.
A subset $K$ of $H^0(X, L)$ is called an {\em arithmetic linear series of $\overline{L}$}
if $K$ is a symmetric convex lattice in $H^0(X, L)$
with 
\[
K \subseteq B_{\sup}(\overline{L}) := \{ s \in H^0(X_\RR, L_{\RR}) \mid \Vert s \Vert_{\sup} \leq 1\}.
\]
If $K = \aH(X, \overline{L})$ $(= H^0(X, L) \cap B_{\sup}(\overline{L}))$,
then $K$ is said to be {\em complete}.
Then we have the following theorem:

\begin{Theorem}
\label{thm:estimate:valuation:vec}
Let $\nu$ be the valuation attached to a good flag $Y_{\cdot} : X \supset Y_1 \supset Y_2 \supset \cdots \supset Y_d$
over a prime $p$.
If $K$ is an arithmetic linear series of $\overline{L}$, then the following estimation
\begin{multline*}
\vert \#\nu(K \setminus \{ 0 \}) \log p - \log \#(K) \vert \\
\leq \left( \log\left(4 p \rank \langle K \rangle_{\ZZ} \right) +
\frac{\sigma(\overline{L}) + \log\left(2 p \rank \langle K \rangle_{\ZZ} \right)}{\log p}  \log(4) \rank H^0(\OO_X) \right)
\rank \langle K \rangle_{\ZZ}
\end{multline*}
holds, 
where $\sigma(\overline{L})$ is given by
\[
\sigma(\overline{L}) := \inf_{\text{$\overline{A}$ : ample}}
\frac{\adeg (\acherncl_1(\overline{A})^{d-1} \cdot \acherncl_1(\overline{L}))}{\deg (A_{\QQ}^{d-1})}.
\]
\end{Theorem}

\begin{proof}
We set $\beta = p \rank  \langle K \rangle_{\ZZ}$, $\Delta = \Conv_{\RR}(K)$ and $M_k =  \langle K \rangle_{\ZZ} \cap H^0(X, L - k Y_1)$ for $k \geq 0$.
Then $M_0 =  \langle K \rangle_{\ZZ}$, $K = M_0 \cap \Delta$,
$\rank M_k = \rank M_0$ and 
\[
M_{k+1} = M_k \cap H^0(Y, (L - kY_1) - Y_1).
\]
Let
$r_k : H^0(X, L-kY_1) \to H^0(Y_1, \rest{L-kY_1}{Y_1})$
be the natural homomorphism
for each $k \geq 0$.
Note that
\[
\#\nu(K \setminus \{ 0 \}) = \sum_{k \geq 0} \#\nu_{Y_1} \left(r_k(M_k \cap \Delta) \setminus \{ 0 \} \right).
\]
Thus, by applying  Proposition~\ref{prop:estimate:val:on:Y:1}
to $L - kY_1$,
we obtain
\begin{multline*}
\sum_{k\geq 0} \left( \log \#(M_k \cap (1/\beta)\Delta) - \log \#(M_{k+1} \cap (2/\beta) \Delta )\right) \\
\leq
\# \nu(K \setminus \{ 0 \}) \log p \leq \\
\sum_{k\geq 0} \left( \log \#(M_k \cap 2\beta \Delta) - \log \#(M_{k+1} \cap \beta \Delta )\right),
\end{multline*}
which implies
\begin{multline*}
\# \nu(K \setminus \{ 0 \}) \log p
\leq \log \#(M_0 \cap 2\beta\Delta)  \\
+ \sum_{k\geq 1} \left( \log \#(M_k \cap 2\beta\Delta) - \log \#(M_k \cap \beta\Delta )\right)
\end{multline*}
and
\begin{multline*}
\# \nu(K \setminus \{ 0 \}) \log p
\geq \log \#(M_0 \cap (1/\beta)\Delta)  \\
- \sum_{k\geq 1} \left( \log \#(M_k \cap (2/\beta)\Delta ) - \log \#(M_k \cap (1/\beta)\Delta )\right).
\end{multline*}
By \eqref{eqn:lem:conv:module:rest:estimate:3},
\[
\log \#(M_0 \cap 2\beta\Delta)
\leq \log \#(K) +  \log (4\beta)
\rank M_0
\]
and
\[
 \log \#(M_k \cap 2\beta\Delta) - \log \#(M_k \cap \beta\Delta) \\
 \leq
 \log(4) \rank M_k = \log(4) \rank M_0.
\]
Note that
$2\beta\Delta \cap M_k \subseteq \aH(\overline{L} - kY_1 + \overline{\OO}(\log(2\beta)))$
(for the definition of $\overline{\OO}(-)$, see Section~\ref{section:aPIc:Q:R}).
Thus, if we set
\[
S = \left\{ k \geq 1 \mid \aH(\overline{L} - kY_1 + \overline{\OO}(\log(2\beta))) \not = \{ 0 \} \right\},
\]
then
\[
\# \nu(K \setminus \{ 0 \}) \log p \leq  \log \#(K) 
+
\left(  \log(4\beta)+ \#S \log(4) \right)
\rank M_0.
\]
Let $\overline{A}$ be an ample $C^{\infty}$-hermitian invertible sheaf on $X$.
If $k \in S$, then
\begin{align*}
0 & \leq \adeg (\acherncl_1(\overline{A})^{d-1} \cdot \acherncl_1(\overline{L} - kY_1 + \overline{\OO}(\log(2\beta)))) \\
& = \adeg(\acherncl_1(\overline{A})^{d-1} \cdot \acherncl_1(\overline{L})) +
\log(2\beta) \deg(A_{\QQ}^{d-1}) - k \frac{\deg(A_{\QQ}^{d-1}) \log p}{\rank H^0(\OO_X)},
\end{align*}
which implies that
\[
k \leq \frac{(\sigma(\overline{L}) + \log(2\beta))\rank H^0(\OO_X)}{\log p},
\]
and hence
\[
\# S \leq \frac{(\sigma(\overline{L}) + \log(2\beta))\rank H^0(\OO_X)}{\log p}.
\]
Further, by using \eqref{eqn:lem:conv:module:rest:estimate:3}, we can see
\[
\log \#(M_0 \cap (1/\beta)\Delta)
\geq \log \# (K) - \log(2\beta )
\rank M_0
\]
and
\[
\log \#(M_k \cap (2/\beta)\Delta) - \log \#(M_k \cap (1/\beta)\Delta)
\leq  \log(4) \rank M_k =  \log(4) \rank M_0.
\]
Hence, as before, we obtain
\[
\# \nu(K \setminus \{ 0 \}) \log p \geq  \log \#(K) 
-
\left(  \log(2\beta)+ \#S \log(4)  \right) \rank M_0,
\]
as required.
\end{proof}

\setcounter{equation}{0}
\begin{Corollary}
\label{thm:main:cor}
There is a positive constant $c = c(X, \overline{L})$ depending only on $X$ and $\overline{L}$
with the following property:
For a good flag 
\[
Y_{\cdot} : X \supset Y_1 \supset Y_2 \supset \cdots \supset Y_d
\]
over a prime $p$,
there is a positive constant $m_0 = m_0\left(p, X_{\QQ}, L_{\QQ}\right)$ depending only on 
$p$, $X_{\QQ}$  and $L_{\QQ}$ such that,
if $m \geq m_0$, then
\[
\vert \#\nu_{Y_{\cdot}}(K \setminus \{ 0 \}) \log p - \log \#(K) \vert
\leq \frac{cm^d}{\log p}
\]
holds for any arithmetic linear series $K$ of $m\overline{L}$,
where $\nu_{Y_{\cdot}}$ is the valuation attached to the flag
$Y_{\cdot} : X \supset Y_1 \supset Y_2 \supset \cdots \supset Y_d$.
\end{Corollary}

\begin{proof}
The problem is an estimation of $C_m$ given by 
\[
\left( \log\left(4 p \rank H^0(mL) \right) +
\frac{\sigma(m\overline{L}) + \log\left(2 p \rank H^0(mL)\right)}{\log p}  \log(4) \rank H^0(\OO_X)\right)
\rank H^0(mL).
\]
First of all, there is a constant $c_1$ depending only on $X_{\QQ}$ and $L_{\QQ}$ such that
\[
\rank H^0(mL) \leq c_1 m^{d-1}
\]
for all $m \geq 0$.
Thus
\begin{multline*}
C_m \leq \frac{1}{m}
 \left( \log\left(4 p c_1 m^{d-1}\right) \log(p)+ 
\log\left(2 p c_1 m^{d-1} \right)\log(4) \rank H^0(\OO_X)\right) \frac{c_1 m^{d}}{\log p} \\
+
\sigma(\overline{L})\log(4) \rank H^0(\OO_X) \frac{c_1 m^{d}}{\log p}.
\end{multline*}
We can find a positive integer $m_0$ depending only on
$p$ and $c_1$ such that
if $m \geq m_0$, then
\[
\frac{1}{m}
 \left( \log\left(4 p c_1 m^{d-1}\right) \log(p)+ 
\log\left(2 p c_1 m^{d-1} \right)\log(4) \rank H^0(\OO_X)\right) \leq 1 + \rank H^0(\OO_X).
\]
Therefore, 
\[
C_m \leq \left(1 + ( 1 +  \sigma(\overline{L})\log(4))\rank H^0(\OO_X) \right)c_1\frac{m^{d}}{\log p}
\]
for $m \geq m_0$, 
as required.
\end{proof}

As an application of Corollary~\ref{thm:main:cor}, we have the following theorem.
The arithmetic Fujita's 
approximation theorem is a straightforward consequence of this result.

\setcounter{equation}{0}
\begin{Theorem}
\label{thm:approximation:k:fold:sum}
Let $\overline{L}$ be a big continuous hermitian invertible sheaf on 
a projective arithmetic variety $X$.
For any positive $\epsilon$, there is a positive integer $n_0 = n_0(\epsilon)$ such that,
for all $n \geq n_0$,
\[
\liminf_{k\to\infty} \frac{
\log\#(K_{k, n})}{n^dk^d}
\geq \frac{\avol(\overline{L})}{d!} - \epsilon,
\]
where $K_{k, n}$ is the convex lattice hull of
\[
V_{k,n} = \{ s_1 \otimes \cdots \otimes s_k \mid s_1, \ldots, s_k \in \aH(X, n \overline{L})\}
\]
in $H^0(X, knL)$ .
\end{Theorem}

\begin{proof}
A generalization of this theorem will be proved in Theorem~\ref{thm:vol:arith:rest:linear:system}.
\end{proof}

\section{Base locus of continuous hermitian invertible sheaf}
Let $X$ be a projective arithmetic variety and
$\overline{L}$ a continuous hermitian invertible sheaf on $X$.
We define the {\em base locus}  $\aBs(\overline{L})$ of $\overline{L}$ to be
\[
\aBs(\overline{L}) = \Supp\left( \Coker\left( \langle \aH(X, \overline{L}) \rangle_{\ZZ} \otimes \OO_X \to L \right) \right),
\]
that is,
\[
\aBs(\overline{L}) = \{ x \in X \mid \text{$s(x) = 0$ for all $s \in \aH(X, \overline{L})$} \}.
\] 
Moreover, the {\em stable base locus} $\aSBs(\overline{L})$ is defined to be
\[
\aSBs(\overline{L}) = \bigcap_{m \geq 1} \aBs(m\overline{L})
\]
The following proposition is the basic properties of $\aBs(\overline{L})$ and $\aSBs(\overline{L})$.

\begin{Proposition}
\label{prop:basic:properties:base:locus}
\begin{enumerate}
\renewcommand{\labelenumi}{\rom{(\arabic{enumi})}}
\item 
$\aBs(\overline{L}+ \overline{M}) \subseteq \aBs(\overline{L})  \cup \aBs(\overline{M})$
for any $\overline{L}, \overline{M} \in \aPic(X;C^0)$.

\item
There is a positive integer $m_0$ such that
$\aSBs(\overline{L}) = \aBs(mm_0 \overline{L})$ for all $m \geq 1$.

\item
$\aSBs(\overline{L}+ \overline{M}) \subseteq \aSBs(\overline{L})  \cup \aSBs(\overline{M})$
for any $\overline{L}, \overline{M} \in \aPic(X;C^0)$.

\item
$\aSBs(\overline{L}) = \aSBs(m\overline{L})$ for all $m \geq 1$.
\end{enumerate}
\end{Proposition}

\begin{proof}
(1) is obvious by its definition.

(2) By using (1), it is sufficient to find a positive integer $m_0$ with
$\aSBs(\overline{L}) = \aBs(m_0 \overline{L})$. Thus it is enough to see that
if $\aSBs(\overline{L}) \subsetneq \aBs(a \overline{L})$, then there is $b$ with
$ \aBs(a b \overline{L}) \subsetneq \aBs(a \overline{L})$.
Indeed, choose $x \in \aBs(a \overline{L}) \setminus \aSBs(\overline{L})$.
Then there is $b$ with $x \not\in \aBs(b \overline{L})$, so that
$x \not\in \aBs(a b \overline{L})$ by (1).

(3) This is a consequence of (1) and (2).

(4) Clearly $\aSBs(\overline{L}) \subseteq \aSBs(m\overline{L})$.
We choose $m_0$ with $\aSBs(\overline{L}) = \aBs(m_0 \overline{L})$.
Then $\aSBs(m\overline{L}) \subseteq \aBs(m_0 m\overline{L}) = \aSBs(\overline{L})$.
\end{proof}

Let $\iota : \aPic(X;C^0) \to \aPic_{\QQ}(X;C^0) (:= \aPic(X;C^0) \otimes \QQ)$ be
the natural homomorphism. For $\overline{L} \in \aPic_{\QQ}(X;C^0)$,
there are a positive integer $n$ and $\overline{M} \in \aPic(X;C^0)$ such that
$\overline{L} = (1/n)\iota(\overline{M})$.
Then, by the above (4), we can see that $\aSBs(\overline{M})$ does not depend on the choice of $n$ and $\overline{M}$,
so that $\aSBs(\overline{L})$ is defined by $\aSBs(\overline{M})$.
The {\em augmented base-locus} $\aSBs_+(\overline{L})$ of $\overline{L}$ is defined to be
\[
\aSBs_+(\overline{L}) = 
\bigcup_{\substack{\overline{A} \in \aPic_{\QQ}(X;C^{\infty}) \\ \text{$\overline{A}$ : ample}}} \aSBs(\overline{L} - \overline{A}).
\]

\begin{Proposition}
\label{prop:augmented:base:locus}
Let $\overline{B}_1, \ldots, \overline{B}_r$ be ample $C^{\infty}$-hermitian $\QQ$-invertible sheaves on $X$.
Then there is a positive number $\epsilon_0$ such that
\[
\aSBs_+(\overline{L}) = \aSBs(\overline{L} - \epsilon_1 \overline{B}_1 - \cdots - \epsilon_r \overline{B}_r)
\]
for all rational numbers $\epsilon_1, \ldots , \epsilon_r$ with
$0 < \epsilon_1\leq \epsilon_0, \ldots , 0 < \epsilon_r \leq \epsilon_0$.
\end{Proposition}

\begin{proof}
Since $X$ is a noetherian space,
there are ample $C^{\infty}$-hermitian $\QQ$-invertible sheaves $\overline{A}_1, \ldots, \overline{A}_l$ on $X$ such that
$\aSBs_+(\overline{L}) = \bigcap_{i=1}^l \aSBs(\overline{L} - \overline{A}_i)$.
We choose a positive number $\epsilon_0$ such that, for all rational numbers $\epsilon_1, \ldots , \epsilon_r$ with
$0 < \epsilon_1\leq \epsilon_0, \ldots , 0 < \epsilon_r \leq \epsilon_0$,
\[
\overline{A}_i - \epsilon_1 \overline{B}_1 - \cdots - \epsilon_r \overline{B}_r
\]
is ample for every $i = 1, \ldots, l$.
Then, by (2) in Proposition~\ref{prop:basic:properties:base:locus},
\begin{align*}
\aSBs(\overline{L} - \epsilon_1 \overline{B}_1 - \cdots - \epsilon_r \overline{B}_r) &
=
\aSBs(\overline{L} - \overline{A}_i + (\overline{A}_i - \epsilon_1 \overline{B}_1 - \cdots - \epsilon_r \overline{B}_r)) \\
& \subseteq \aSBs(\overline{L} - \overline{A}_i) \cup \aSBs(\overline{A}_i - \epsilon_1 \overline{B}_1 - \cdots - \epsilon_r \overline{B}_r) \\
& = \aSBs(\overline{L} - \overline{A}_i),
\end{align*}
which implies
\[
\aSBs_+(\overline{L}) \subseteq \aSBs(\overline{L} - \epsilon_1 \overline{B}_1 - \cdots - \epsilon_r \overline{B}_r) \subseteq
\bigcap_{i=1}^l \aSBs(\overline{L} - \overline{A}_i) = \aSBs_+(\overline{L}).
\]
\end{proof}

\section{Arithmetic Picard group and cones}
\label{section:aPIc:Q:R}
According as \cite{MoExt}, we fix several notations.
Let $X$ be a projective arithmetic variety.
Let $C^0(X)$ be the set of real valued continuous functions $f$ on $X(\CC)$ with
$F^*_{\infty}(f) = f$, where
$F_{\infty}  : X(\CC) \to X(\CC)$ is the complex conjugation map on $X(\CC)$.
Let
$\overline{\OO} : C^0(X) \to \aPic(X;C^0)$ be the homomorphism given by
\[
\overline{\OO}(f) = (\OO_X, \exp(-f)\vert\cdot\vert_{can}).
\]
We define $\aPic_{\QQ}(X;C^0)$ and $\aPic_{\otimes \RR}(X;C^0)$ to be 
\[
\aPic_{\QQ}(X;C^0) := \aPic(X;C^0) \otimes \QQ\quad\text{and}\quad
\aPic_{\otimes \RR}(X;C^0) := \aPic(X;C^0) \otimes \RR.
\]
We denote the natural homomorphism $\aPic(X;C^0) \to \aPic_{\QQ}(X;C^0)$ by $\iota$.
Let $N(X)$ be the subgroup of $\aPic_{\otimes \RR}(X;C^0)$ consisting of elements
\[
\overline{\OO}(f_1) \otimes x_1 + \cdots + \overline{\OO}(f_r) \otimes x_r
\quad(f_1,\ldots,f_r \in C^0(X),\ x_1,\ldots,x_r \in \RR)
\]
with $x_1f_1+\cdots+x_rf_r = 0$.
We define $\aPic_{\RR}(X;C^0)$ to be
\[
\aPic_{\RR}(X;C^0) :=\aPic_{\otimes \RR}(X)/N(X).
\]
Let $\pi : \aPic_{\otimes \RR}(X;C^0)  \to \aPic_{\RR}(X;C^0)$ be the natural homomorphism. 
Here we give the strong topology to $\aPic_{\QQ}(X;C^0)$, $\aPic_{\otimes \RR}(X;C^0)$ and $\aPic_{\RR}(X;C^0)$.
Then the homomorphisms 
\[
\aPic_{\QQ}(X;C^0) \hookrightarrow \aPic_{\otimes \RR}(X;C^0)\quad\text{and}\quad
\pi : \aPic_{\otimes \RR}(X;C^0)  \to \aPic_{\RR}(X;C^0)
\]
are continuous. 
Moreover, $\pi : \aPic_{\otimes \RR}(X;C^0)  \to \aPic_{\RR}(X;C^0)$ is an open map (cf. Conventions and terminology~\ref{CT:strong:topology}).
We denote the composition of homomorphisms
\[
\aPic_{\QQ}(X;C^0) \hookrightarrow
\aPic_{\otimes \RR}(X;C^0)  \overset{\pi}{\longrightarrow} \aPic_{\RR}(X;C^0)
\]
by $\rho$.
Then $\rho$ is also continuous.
Note that $\rho$ is not necessarily injective (cf. \cite[Example~4.5]{MoExt}).

\bigskip
Let $\aAmp(X)$ be the sub-semigroup of $\aPic(X;C^0)$ consisting of all ample $C^{\infty}$-hermitian invertible sheaves on $X$.
Let us observe the following lemma.

\begin{Lemma}
\label{lem:ample:in:ample:f:>=0}
Let $\overline{A}$ be an ample invertible sheaf on $X$. For any $\overline{L} \in \aPic(X;C^0)$,
there are  a positive integer $n_0$ and $f \in C^0(X)$ such that $f \geq 0$ and
\[
\overline{L} + n \overline{A} - \overline{\OO}(f) \in \aAmp(X)
\]
for all $n \geq n_0$.
\end{Lemma}

\begin{proof}
Let $\vert\cdot\vert$ be the hermitian metric of $\overline{L}$ and
$\vert\cdot\vert_0$ a $C^{\infty}$-hermitian metric of $L$.
We set $\vert\cdot\vert = \exp(-f_0) \vert\cdot\vert_0$ for some $f_0 \in C^0(X)$.
We can take a constant $c$ with $f_0 +c \geq 0$, and put $f = f_0 + c$.
Then $f \geq 0$ and $\exp(f) \vert\cdot\vert$ is $C^{\infty}$, which means that
$\overline{L} - \overline{\OO}(f)$ is $C^{\infty}$.
Thus there is a positive integer $n_0$ such that
\[
(\overline{L} - \overline{\OO}(f))+ n \overline{A}  \in \aAmp(X)
\]
for all $n \geq n_0$.
\end{proof}

\begin{Proposition}
\label{prop:big:open:etc:arithmetic:Pic}
Let $\widehat{C}$ be a sub-monoid of $\aPic(X;C^0)$ such that
\[
\{ \overline{\OO}(f) \mid f \in C^0(X), \ f \geq 0 \} \subseteq \widehat{C}.
\]
We set $\widehat{B} = \QSat(\aAmp(X) + \widehat{C})$ \rom{(}cf. see Conventions and terminology~\rom{\ref{CT:sub:semigroup:monoid}} for
the saturation\rom{)}. Then we have the following.
\begin{enumerate}
\renewcommand{\labelenumi}{\rom{(\arabic{enumi})}}
\item
$\widehat{B}$ is open, that is,
for any $\overline{L} \in \widehat{B}$ and $\overline{M} \in \aPic(X;C^0)$,
there is a positive integer $n$ such that $n \overline{L} + \overline{M} \in \widehat{B}$.

\item
If we set
\[
\begin{cases}
\widehat{B}_{\QQ} := \Cone_{\QQ}(\iota(\widehat{B})) &  \text{in $\aPic_{\QQ}(X;C^0)$}, \\
\widehat{B}_{\otimes \RR} := \Cone_{\RR}(\widehat{B}_{\QQ}) & \text{in $\aPic_{\otimes \RR}(X;C^0)$}, \\
\widehat{B}_{\RR} := \Cone_{\RR}(\rho(\widehat{B}_{\QQ})) & \text{in $\aPic_{\RR}(X;C^0)$},
\end{cases}
\]
then $\widehat{B}_{\QQ}$, $\widehat{B}_{\otimes \RR}$ and $\widehat{B}_{\RR}$ are open in 
$\aPic_{\QQ}(X;C^0)$, $\aPic_{\otimes \RR}(X;C^0)$ and $\aPic_{\RR}(X;C^0)$ respectively.

\item
\[
\begin{cases}
\iota^{-1}\left( \widehat{B}_{\QQ} \right) = \widehat{B}, \quad \widehat{B}_{\otimes \RR} \cap \aPic_{\QQ}(X;C^0) = \widehat{B}_{\QQ}, \\
\pi^{-1}\left(\widehat{B}_{\RR}\right) = \widehat{B}_{\otimes \RR}, \quad \rho^{-1}\left(\widehat{B}_{\RR}\right)= \widehat{B}_{\QQ}.
\end{cases}
\]
\end{enumerate}
\end{Proposition}

\begin{proof}
(1) Let $\overline{L} \in \widehat{B}$ and $\overline{M} \in \aPic(X;C^0)$.
Then there is a positive integer $n_0$ such that $n_0\overline{L} = \overline{A} + \overline{E}$ for some
$\overline{A} \in \aAmp(X)$ and $\overline{E} \in \widehat{C}$.
By Lemma~\ref{lem:ample:in:ample:f:>=0},
there are  a positive integer $n_1$ and $f \in C^0(X)$ such that $f \geq 0$ and
$\overline{M} + n_1 \overline{A} - \overline{\OO}(f) = \overline{A}'$
for some $\overline{A}' \in \aAmp(X)$.
Then
\[
n_1 n_0 \overline{L} + \overline{M} = n_1(\overline{A} + \overline{E}) + \overline{M} = \overline{A}' + (n_1\overline{E} + \overline{\OO}(f)) \in \widehat{B}.
\]

(2) follows from (3) in Proposition~\ref{prop:QSat},
(2) in Proposition~\ref{prop:A:V:open} and 
(2.1) in Proposition~\ref{prop:big:open}.

(3) 
Let us consider the following claim:

\begin{Claim}
$\widehat{B}_{\otimes \RR} + N(X) \subseteq \widehat{B}_{\otimes \RR}$.
\end{Claim}

First of all, let us see the following formula:
\addtocounter{Claim}{1}
\begin{equation}
\label{eqn:prop:big:open:etc:arithmetic:Pic:q}
 \widehat{B}_{\otimes \RR} + \iota(\widehat{C}) \subseteq \widehat{B}_{\otimes \RR}.
\end{equation}
Indeed, as $\widehat{B} + \widehat{C} \subseteq \widehat{B}$,
we have $\iota(\widehat{B} + \widehat{C}) \subseteq \iota(\widehat{B})$. Thus,
by (3) in Proposition~\ref{prop:A:V:open}, 
\[
\widehat{B}_{\otimes \RR} + \iota(\widehat{C}) \subseteq \Cone_{\RR}(\iota(\widehat{B})) + \Cone_{\RR}(\iota(\widehat{C}))
= \Cone_{\RR}(\iota(\widehat{B} + \widehat{C})) \subseteq  \widehat{B}_{\otimes \RR}.
\]

Let $a \in \widehat{B}_{\otimes \RR}$ and $x \in N(X)$.
We set $x = \overline{\OO}(f_1)\otimes a_1 + \cdots + \overline{\OO}(f_r) \otimes a_r$ with
$a_1 f_1 + \cdots + a_r f_r = 0$, where $f_1, \ldots, f_r \in C^0(X)$ and
$a_1, \ldots, a_r \in \RR$.
Let us take a sequence $\{ a_{in} \}_{n=1}^{\infty}$ in $\QQ$ such that
$a_i = \lim_{n\to\infty} a_{in}$. We set $\phi_n = a_{1n} f_1 + \cdots + a_{rn}f_r$.
Then
\begin{multline*}
\Vert \phi_n \Vert_{\sup} = \Vert (a_{1n} - a_1) f_1 + \cdots + (a_{rn}-a_r) f_r \Vert_{\sup} \\
\leq
\vert a_{1n} - a_1\vert \Vert f_1  \Vert_{\sup} + \cdots + \vert a_{rn}-a_r\vert  \Vert f_r \Vert_{\sup}.
\end{multline*}
Thus $\lim_{n\to\infty} \Vert \phi_n \Vert_{\sup} = 0$.
We choose a sequence $\{ b_n \}$ in $\QQ$ such that
$b_n \geq  \Vert \phi_n \Vert_{\sup}$ and $\lim_{n\to\infty} b_n = 0$.
Then $\phi_n + b_n \geq 0$.
If we put 
\[
x_n = \overline{\OO}(f_1)\otimes a_{1n} + \cdots + \overline{\OO}(f_r) \otimes a_{rn} + \overline{\OO}(1) \otimes b_n,
\]
then $\lim_{n\to \infty} x_n = x$.
On the other hand, as $x_n = \overline{\OO}(\phi_n + b_n )$ in $\aPic_{\QQ}(X;C^0)$, $x_n \in \iota(\widehat{C})$.
By (2), $\widehat{B}_{\otimes \RR}$ is an open set in $\aPic_{\otimes \RR}(X;C^0)$.
Thus, if $n \gg 1$, then
$(x - x_n) + a \in \widehat{B}_{\otimes \RR}$.
Hence the claim follows because
\[
x + a =  \left( (x - x_n) + a \right) + x_n \in \widehat{B}_{\otimes \RR} + \iota(\widehat{C}) \subseteq \widehat{B}_{\otimes \RR}.
\]
\CQED

The first formula follows from (2) in Proposition~\ref{prop:QSat}.
The second is derived from (1) in Proposition~\ref{prop:A:V:open}. We can see the third by using
(2.2) in Proposition~\ref{prop:big:open} and the above claim.
The last formula follows from the second and the third.
\end{proof}

\section{Big hermitian invertible sheaves with respect to \\an arithmetic subvariety}
Let $X$ be a projective arithmetic variety and 
$Y$ an arithmetic subvariety of $X$, that is, $Y$ is an integral subscheme of $X$ such that $Y$ is flat over $\Spec(\ZZ)$.
A continuous hermitian invertible sheaf $\overline{L}$ is said to be {\em $Y$-effective} (or {\em effective with respect to $Y$})
if there is $s \in \aH(X, \overline{L})$ with $\rest{s}{Y} \not= 0$.
For $\overline{L}_1, \overline{L}_2 \in \aPic(X)$,
if $\overline{L}_1 - \overline{L}_2$ is $Y$-effective, then
we denote it by $\overline{L}_1 \geq_Y \overline{L}_2$.
We define $\aEff(X;Y)$ to be
\[
\aEff(X;Y) := \left\{ \overline{L} \in \aPic(X;C^0) \mid \text{$\overline{L}$ is $Y$-effective}\right\}.
\]
Then it is easy to see the following.
\begin{enumerate}
\renewcommand{\labelenumi}{\rom{(\alph{enumi})}}
\item
$\aEff(X;Y)$ is a sub-monoid of $\aPic(X;C^0)$.

\item
$\{ \overline{\OO}(f)  \mid f \in C^0(X),\ f \geq 0\} \subseteq \aEff(X;Y)$.
\end{enumerate}
Here we define $\aBig(X;Y)$, $\aBig_{\QQ}(X;Y)$, $\aBig_{\otimes \RR}(X;Y)$ and $\aBig_{\RR}(X;Y)$ to be
\[
\begin{cases}
\aBig(X;Y) := \QSat(\aAmp(X) + \aEff(X;Y)), \\
\aBig_{\QQ} (X;Y) := \Cone_{\QQ}(\iota(\aBig(X;Y))), \\
\aBig_{\otimes \RR} (X;Y) := \Cone_{\RR}(\aBig_{\QQ}(X;Y)), \\
\aBig_{\RR} (X;Y) := \Cone_{\RR}(\rho(\aBig_{\QQ}(X;Y))),
\end{cases}
\]
where $\iota$, $\pi$ and $\rho$ are the natural homomorphisms as follows:
\[
\xymatrix{
\aPic(X;C^0)  \ar[r]^(.45){\iota}  \ar[rrd]_{\gamma} & \aPic_{\QQ}(X;C^0) \ar@{^{(}->}[r] \ar[dr]^{\rho} & \aPic_{\otimes \RR}(X;C^0)  \ar[d]^{\pi} \\
& & \aPic_{\RR}(X;C^0) \\
}
\]
For the definition of the saturation, 
see Conventions and terminology~\ref{CT:sub:semigroup:monoid}.
By Proposition~\ref{prop:big:open:etc:arithmetic:Pic},
$\aBig_{\QQ}(X;Y)$, $\aBig_{\otimes \RR}(X;Y)$ and $\aBig_{\RR}(X;Y)$ are open in 
$\aPic_{\QQ}(X;C^0)$, $\aPic_{\otimes \RR}(X;C^0)$ and $\aPic_{\RR}(X;C^0)$ respectively.
Moreover,
\[
\begin{cases}
\iota^{-1}\left( \aBig_{\QQ}(X;Y) \right) = \aBig(X;Y), \  \aBig_{\otimes \RR}(X;Y)  \cap \aPic_{\QQ}(X;C^0) = \aBig_{\QQ}(X;Y), \\
\pi^{-1}\left(\aBig_{\RR}(X;Y)\right) = \aBig_{\otimes \RR}(X;Y), \  \rho^{-1}\left(\aBig_{\RR}(X;Y) \right)= \aBig_{\QQ}(X;Y).
\end{cases}
\]
A continuous hermitian invertible sheaf $\overline{L}$ on $X$ is said to 
be {\em $Y$-big} (or {\em big with respect to $Y$})
if $\overline{L} \in \aBig(X;Y)$.
In the remaining of this section, we will observe several basic properties of
$Y$-big continuous hermitian invertible sheaves.
Let us begin with the following proposition.

\begin{Proposition}
\begin{enumerate}
\renewcommand{\labelenumi}{\rom{(\arabic{enumi})}}
\item
Let $\overline{L}$ be a continuous hermitian invertible sheaf  on $X$. Then
the following are equivalent:
\begin{enumerate}
\renewcommand{\labelenumii}{\rom{(\arabic{enumi}.\arabic{enumii})}}
\item
$\overline{L}$ is $Y$-big.

\item
For any $\overline{A} \in \aAmp(X)$,
there is a positive integer $n$ with
$n \overline{L} \geq_{Y} \overline{A}$.

\item
$Y \not\subseteq \aSBs_+(\overline{L})$.
\end{enumerate}

\item
If $\overline{L}$ is $Y$-big, then there is a positive integer $m_0$ such that
$m \overline{L}$ is $Y$-effective for all $m \geq m_0$.

\end{enumerate}
\end{Proposition}

\begin{proof}
(1) (1.1) $\Longrightarrow$ (1.2) :
There is a positive integer $n$ such that
$n\overline{L} = \overline{B} + \overline{M}$ for some $\overline{B} \in \aAmp(X)$ and $\overline{M} \in \aEff(X;Y)$.
Let $\overline{A}$ be an ample $C^{\infty}$-hermitian invertible sheaf on $X$. 
We choose a positive number $n_1$ such that $n_1\overline{B} - \overline{A}$ is $Y$-effective.
Then
\[
n_1n\overline{L} - \overline{A} = (n_1\overline{B} - \overline{A}) + n_1\overline{M}
\]
is $Y$-effective. 

(1.2) $\Longrightarrow$ (1.3) :
For an ample $C^{\infty}$-hermitian invertible  $\overline{A}$ sheaf,
there is a positive integer $n$  such that $n\overline{L} \geq_{Y} \overline{A}$.
Thus there is $s \in \aH(X, n\overline{L} - \overline{A})$ with $\rest{s}{Y} \not= 0$, which means that
$Y \not\subseteq \aBs(n\overline{L} -  \overline{A})$.
Note that
\[
\aBs(n\overline{L} -  \overline{A}) \supseteq \aSBs(n \overline{L} - \overline{A}) = \aSBs(\overline{L} - (1/n)\overline{A})
\supseteq \aSBs_+(\overline{L}).
\]
Hence $Y \not\subseteq \aSBs_+(\overline{L})$.

(1.3) $\Longrightarrow$ (1.1) :
Let $\overline{A}$ be an ample $C^{\infty}$-hermitian invertible sheaf.
Then, by Proposition~\ref{prop:augmented:base:locus}, 
there is a positive number $n$ such that
\[
\aSBs_+(\overline{L}) = \aSBs(\overline{L} - (1/n)\overline{A}) = \aSBs(n \overline{L} - \overline{A}).
\]
Thus, by (2) in Proposition~\ref{prop:basic:properties:base:locus},
we can find a positive integer $m$ such that 
\[
\aSBs_+(\overline{L}) = \aBs(m(n \overline{L} - \overline{A})),
\]
so that
there is $s \in \aH(X, m(n \overline{L} - \overline{A}))$ with $\rest{s}{Y} \not= 0$ because $Y \not\subseteq \aSBs_+(\overline{L})$.
This means that $mn \overline{L} \geq_{Y} m\overline{A}$, as required.

(2)
We choose an ample $C^{\infty}$-hermitian invertible sheaf $\overline{A}$ such that
$\overline{A}$ and $\overline{L} + \overline{A}$ is $Y$-effective.
Moreover, we can take a positive integer $a$ such that $a\overline{L} - \overline{A}$ is $Y$-effective
because $\overline{L}$ is $Y$-big.
Note that $a\overline{L} = (a \overline{L} - \overline{A}) + \overline{A}$ and
$(a+1)\overline{L} = (a \overline{L} - \overline{A}) + (\overline{L} + \overline{A})$.
Thus $a\overline{L}$ and
$(a+1)\overline{L}$ are $Y$-effective.
Let $m$ be an integer with $m \geq a^2 + a$.
We set $m = aq + r$ ($0 \leq r < a$).
Then $q \geq a$, so that there is an integer $b$ with $q = b+r$ and $b > 0$.
Therefore, $m \overline{L}$ is $Y$-effective because
$m \overline{L} = b(a\overline{L}) + r ((a+1)\overline{L})$.
\end{proof}

\begin{Proposition}
\label{prop:Y:big:span:Okounkov:body}
Let $X$ be a projective arithmetic variety, $Y$ a $d'$-dimensional arithmetic subvariety of $X$ and $\overline{L}$ a continuous hermitian invertible sheaf on $X$.
Let $Z_{\cdot} : Z_0 = Y \supset Z_1 \supset Z_2 \supset \cdots \supset Z_{d'}$ be a good flag over a prime $p$ on $Y$.
If $\overline{L}$ is $Y$-big, then
\[
\left\{ \left(\nu_{Z_{\cdot}} \left(\rest{s}{Y}\right), m\right) \mid \text{$s \in \aH(X, m\overline{L})$ and $\rest{s}{Y} \not= 0$}\right\}
\]
generates $\ZZ^{d'+1}$ as a $\ZZ$-module.
\end{Proposition}

To prove the above proposition, we need the following two lemmas.

\begin{Lemma}
\label{lem:exist:section:zero:non:zero}
Let $X$ be either 
a projective arithmetic variety or a projective variety over a field.
Let $Z$ be a reduced and irreducible subvariety of codimension $1$ and $x$ a closed point of $Z$.
Let $I$ be the defining ideal sheaf of $Z$.
We assume that $I$ is principal at $x$ \rom{(}it holds if $X$ is regular at $x$\rom{)}.
Let $H$ be an ample invertible sheaf on $X$.
Then there is a positive integer $n_0$ such that, for all $n \geq n_0$, we can find $s \in H^0(X, nH \otimes I)$ such that
$s \not= 0$ in $nH \otimes I \otimes \kappa(x)$, where $\kappa(x)$ is the residue field at $x$.
\end{Lemma}

\begin{proof}
Let $m_x$ be the maximal ideal at $x$.
Since $I$ is invertible around $x$,
we have the exact sequence
\[
0 \to nH \otimes I \otimes m_x \to nH \otimes I \to nH \otimes I \otimes\kappa(x) \to 0.
\]
As $H$ is ample, there is a positive integer $n_0$ such that
\[
H^1(X, nH \otimes I \otimes m_x) = 0
\]
for all $n \geq n_0$, which means that $H^0(X, nH \otimes I) \to  nH \otimes I \otimes \kappa(x)$ is surjective, as required.
\end{proof}

\begin{Lemma}
\label{lem:ample:valuation:vector:standard:basis}
Let $X$ be a projective arithmetic variety and $Y$ a $d'$-dimensional arithmetic subvariety of $X$.
Let $Z_{\cdot} : Z_0 = Y \supset Z_1 \supset Z_2 \supset \cdots \supset Z_{d'}$ be a good flag over a prime $p$ on $Y$.
Let $H$ be an ample invertible sheaf on $X$.
Let $e_1, \ldots, e_{d'}$ be the standard basis of $\ZZ^{d'}$.
Then there is a positive integer $n_0$ such that, for all $n \geq n_0$,
we can find $s_1, \ldots, s_{d'} \in H^0(X, nH)$ with $\nu_{Z_{\cdot}}(\rest{s_1}{Y}) = e_1,
\ldots, \nu_{Z_{\cdot}}(\rest{s_{d'}}{Y}) = e_{d'}$.
\end{Lemma}

\begin{proof}
First of all, we can find $n'_0$ such that, for all $n \geq n'_0$,
\[
H^0(X, nH) \to H^0(Z_i, \rest{nH}{Z_i})
\]
are surjective for all $i$.
We set $Z_{d'} = \{ z \}$.
For $i =1, \ldots, d'$, let $I_i$ be the defining ideal sheaf of $Z_i$ in $Z_{i-1}$.
Then, by Lemma~\ref{lem:exist:section:zero:non:zero}, there is a positive integer $n'_i$ such that,
for all $n \geq n'_i$, we can find $s'_i \in H^0(Z_i, \rest{nH}{Z_{i-1}} \otimes I_i)$ such that
$s'_i \not= 0$ in $\rest{nH}{Z_{i-1}} \otimes I_i \otimes \kappa(z)$.
Thus, if $n \geq \max\{ n'_0, n'_1, \ldots, n'_{d'}\}$,
then there are $s_1, \ldots, s_{d'} \in H^0(X, nH)$ such that
$\rest{s_i}{Y_{i-1}} = s'_i$ for $i=1, \ldots, d'$.
By our construction, it is easy to see that
$\nu_{Z_{\cdot}}(\rest{s_i}{Y}) = e_i$.
\end{proof}

\begin{proof}[The proof of Proposition~\ref{prop:Y:big:span:Okounkov:body}]
Let us begin with the following claim:

\begin{Claim}
There are an ample  $C^{\infty}$-hermitian invertible sheaf $\overline{A}$ and
$s_0, s_1, \ldots, s_{d'} \in \aH(X, \overline{A})\setminus \{ 0\}$ and $t \in \aH(X, \overline{A} + \overline{L})\setminus \{ 0\}$
such that 
\[
\rest{s_0}{Y} \not= 0, \rest{s_1}{Y} \not= 0, \ldots, \rest{s_{d'}}{Y} \not= 0, \rest{t}{Y} \not= 0
\]
and
\[
\nu_{Z_{\cdot}}(\rest{s_0}{Y}) = 0,\ \nu_{Z_{\cdot}}(\rest{s_1}{Y}) = e_1,\  \ldots,\  \nu_{Z_{\cdot}}(\rest{s_{d'}}{Y}) = e_{d'}
\ \text{and}\ 
\nu_{Z_{\cdot}}(\rest{t}{Y}) = 0.
\]
\end{Claim}

Let $B$ be an ample invertible sheaf on $X$.
By Lemma~\ref{lem:ample:valuation:vector:standard:basis},
there are positive integer $n$,
$s_0, s_1, \ldots, s_d \in H^0(X, nB)\setminus \{ 0\}$ and $t \in H^0(X, nB + L)\setminus \{ 0\}$ such that
\[
\nu_{Z_{\cdot}}(\rest{s_0}{Y}) = 0,\ \nu_{Z_{\cdot}}(\rest{s_1}{Y}) = e_1,\  \ldots,\  \nu_{Z_{\cdot}}(\rest{s_d}{Y}) = e_d
\ \text{and}\ 
\nu_{Z_{\cdot}}(\rest{t}{Y}) = 0.
\]
We choose a $C^{\infty}$-hermitian metric of $B$ such that
$\overline{B}$ is ample,
$s_0, s_1, \ldots, s_d \in \aH(X, n\overline{B})$ and $t \in \aH(X, n\overline{B} + \overline{L})$.
\CQED

Let $M$ be the $\ZZ$-submodule generated by
\[
\{ \left(\nu_{Z_{\cdot}}(\rest{s}{Y}), m\right) \mid \text{$s \in \aH(X, m\overline{L})$ and $\rest{s}{Y} \not= 0$}\}.
\]
Since $\overline{L}$ is $Y$-big, there is a positive integer $a$ with $a \overline{L} \geq_Y \overline{A}$, that is,
there is $e \in \aH(X, a \overline{L} - \overline{A})$ with $\rest{e}{Y} \not= 0$.
Note that 
\[
t \otimes e \in \aH(X, (a + 1)\overline{L})\quad\text{and}\quad
s_0 \otimes e \in \aH(X, a\overline{L}).
\]
Moreover $\nu_{Z_{\cdot}}\left(\rest{t \otimes e}{Y}\right) = \nu_{Z_{\cdot}}(\rest{e}{Y})$ and
$\nu_{Z_{\cdot}}(\rest{s_0 \otimes e}{Y}) = \nu_{Z_{\cdot}}(\rest{e}{Y})$.
Thus 
\[
\left(\nu_{Z_{\cdot}}\left(\rest{t \otimes e}{Y}\right), a+1\right) - \left(\nu_{Z_{\cdot}}(\rest{s_0 \otimes e}{Y}), a\right) = (0, \ldots, 0, 1) \in M.
\]
Further,  as  $s_i \otimes e, s_0 \otimes e \in \aH(X, a \overline{L})$, we obtain
\begin{multline*}
\left(\nu_{Z_{\cdot}}\left(\rest{s_i  \otimes e}{Y}\right), m\right) - (\nu_{Z_{\cdot}}(\rest{s_0 \otimes e}{Y}), m) \\
= (e_i + \nu_{Z_{\cdot}}(\rest{e}{Y}), m) - (\nu_{Z_{\cdot}}(\rest{e}{Y}), m) = (e_i, 0) \in M.
\end{multline*}
Hence $M = \ZZ^{d+1}$.
\end{proof}

\section{Arithmetic restricted volume}

\newcommand{\anony}{\bullet}

Let $X$ be a projective arithmetic variety and $Y$ a $d'$-dimensional arithmetic subvariety  of $X$.
For an invertible sheaf $L$ on $X$,
$\Image(H^0(X, L) \to H^0(Y, \rest{L}{Y}))$ is denoted by $H^0(X|Y, L)$.
We assign an arithmetic linear series 
$\aH_{\anony}(X|Y, \overline{L})$ of $\rest{\overline{L}}{Y}$ 
to each continuous hermitian invertible sheaf $\overline{L}$ on $X$
with the following properties:
\begin{enumerate}
\renewcommand{\labelenumi}{\rom{(\arabic{enumi})}}
\item
$\Image ( \aH(X, \overline{L}) \to H^0(X|Y, L)) \subseteq \aH_{\anony}(X|Y, \overline{L})$. 

\item 
$s \otimes s' \in \aH_{\anony}(X|Y, \overline{L}+\overline{M})$ for all $s \in \aH_{\anony}(X|Y, \overline{L})$ and $s' \in \aH_{\anony}(X|Y, \overline{M})$.
\end{enumerate}
This correspondence $\overline{L} \mapsto \aH_{\anony}(X|Y, \overline{L})$ is called an 
{\em assignment of arithmetic restricted linear series from $X$ to $Y$}.
As examples, we have the following:

\medskip
$\bullet$
$\aH_{\CL}(X|Y, \overline{L})$ :
$\aH_{\CL}(X|Y, \overline{L})$ is the convex lattice hull of
\[
\Image(\aH(X, \overline{L}) \to H^0(X|Y, L))
\]
in $H^0(X|Y, L)$.
This is actually an assignment of arithmetic restricted linear series from $X$ to $Y$.
The above property (1) is obvious. For (2),
let $s_1, \ldots, s_r \in \Image(\aH(X, \overline{L}) \to H^0(X|Y, L))$ and
$s'_1, \ldots, s'_{r'} \in \Image(\aH(X, \overline{M}) \to H^0(X|Y, M))$, and let
\[
\lambda_1, \ldots, \lambda_r\quad\text{and}\quad\lambda'_1, \ldots, \lambda'_{r'}
\]
be non-negative real numbers
with $\lambda_1 + \cdots + \lambda_r = 1$ and $\lambda'_1 + \cdots + \lambda'_{r'} = 1$.
Then
\[
(\lambda_1 s_1 + \cdots + \lambda_r s_r) \otimes (\lambda'_1 s'_1 + \cdots + \lambda'_{r'} s'_{r'}) =
\sum_{i,j} \lambda_i \lambda'_j (s_i \otimes s_j)
\]
and
\[
 \sum_{i,j} \lambda_i \lambda'_j  = (\lambda_1 + \cdots + \lambda_r) (\lambda'_1 + \cdots + \lambda'_{r'}) = 1,
\]
as required.

\medskip
$\bullet$
$\aH_{\quot}(X|Y, \overline{L})$ :
Let $\Vert\cdot\Vert^{X|Y}_{\sup,\quot}$ be the quotient norm of $H^0(X|Y, L)$ induced by
the norm $\Vert\cdot\Vert_{\sup}$ on $H^0(X, L)$ and the natural surjective homomorphism
$H^0(X, L) \to H^0(X|Y, L)$. Then $\aH_{\quot}(X|Y, \overline{L})$ is defined to be
\[
\aH_{\quot}(X|Y, \overline{L}) = \{ s \in H^0(X|Y, L) \mid \Vert s \Vert^{X|Y}_{\sup,\quot} \leq 1 \}.
\]
This is obviously an assignment of arithmetic restricted linear series from $X$ to $Y$.

\medskip
$\bullet$
$\aH_{\sub}(X|Y, \overline{L})$ : 
Let $\Vert \cdot\Vert_{Y,\sup}$ be the norm on $H^0(Y, \rest{L}{Y})$ given by
$\Vert s \Vert_{Y,\sup} = \sup_{y \in Y(\CC)} \vert s\vert(y)$. Let
$\Vert \cdot\Vert^{X|Y}_{\sup,\sub}$ be the sub-norm of $H^0(X|Y, L)$ induced by
$\Vert \cdot\Vert_{Y,\sup}$ on $H^0(Y, \rest{L}{Y})$ and the natural injective homomorphism
$H^0(X|Y, L) \hookrightarrow H^0(Y, \rest{L}{Y})$. Then $\aH_{\sub}(X|Y, \overline{L})$ is defined to be
\[
\aH_{\sub}(X|Y, \overline{L}) = \left\{ s \in H^0(X|Y, L) \mid \Vert s \Vert^{X|Y}_{\sup,\sub} \leq 1 \right\}.
\]
This is obviously an assignment of arithmetic restricted linear series from $X$ to $Y$.

\medskip
Note that
\[
\aH_{\CL}(X|Y, \overline{L}) \subseteq \aH_{\quot}(X|Y, \overline{L}) \subseteq \aH_{\sub}(X|Y, \overline{L})
\]
for any continuous hermitian invertible sheaf $\overline{L}$.
An assignment $\overline{L} \mapsto \aH_{\anony}(X|Y, \overline{L})$ of arithmetic restricted linear series from $X$ to $Y$ is said to be
{\em proper} if, for each $\overline{L} \in \aPic(X,C^0)$,
there is a symmetric and bounded convex set $\Delta$ in $H^0(X|Y, L) \otimes \RR$ such that
$\aH_{\anony}(X|Y, \overline{L} + \overline{\OO}(\lambda)) = H^0(X|Y, L) \cap \exp(\lambda) \Delta$ for all $\lambda \in \RR$.
For example, the assignments $\overline{L} \mapsto \aH_{\quot}(X|Y, \overline{L})$ and
$\overline{L} \mapsto \aH_{\sub}(X|Y, \overline{L})$ are proper.

Let us fix an assignment $\overline{L} \mapsto \aH_{\anony}(X|Y, \overline{L})$ of arithmetic restricted linear series from $X$ to $Y$.
Then we define the restricted arithmetic volume with respect to the assignment to be
\[
\avol_{\anony}\left(X|Y, \overline{L}\right) := \limsup_{m\to\infty} \frac{\log \# \aH_{\anony}(X|Y, m\overline{L})}{m^{d'}/d'!}.
\]
Let us begin with the following proposition.

\begin{Proposition}
\label{prop:basic:properties:rest:vol}
\begin{enumerate}
\renewcommand{\labelenumi}{\rom{(\arabic{enumi})}}
\item
If $\overline{L} \leq_Y \overline{M}$, then
$\# \aH_{\anony}(X|Y, \overline{L}) \leq \# \aH_{\anony}(X|Y, \overline{M})$.
In particular, $\avol_{\anony}\left(X|Y, \overline{L}\right) \leq \avol_{\anony}\left(X|Y, \overline{M} \right)$.

\item
We assume that the assignment $\overline{L} \mapsto \aH_{\anony}(X|Y, \overline{L})$ is proper.
Then, for any $\overline{L} \in \aPic(X;C^0)$ and $f \in C^0(X)$,
\[
\left| \avol_{\anony}\left(X|Y, \overline{L} + \overline{\OO}(f) \right) - \avol_{\anony}\left(X|Y, \overline{L}\right) \right| \leq
d' \vol(X_{\QQ}|Y_{\QQ}, L_{\QQ}) \Vert f \Vert_{\sup},
\]
where $\vol(X_{\QQ}|Y_{\QQ}, L_{\QQ})$ is the algebraic restricted volume \rom{(}cf. \cite{ELMNP}\rom{)}.
\end{enumerate}
\end{Proposition}

\begin{proof}
(1) Let us choose $t \in \aH(X, \overline{M} -\overline{L})$ with $\rest{t}{Y} \not= 0$.
Then $\rest{t}{Y} \in \aH_{\anony}(X|Y, \overline{M}-\overline{L})$ and
\[
s \otimes (\rest{t}{Y}) \in \aH_{\anony}(X|Y, \overline{M})
\]
for any $s \in \aH_{\anony}(X|Y, \overline{L})$, which means that
we have the injective map 
\[
\aH_{\anony}(X|Y, \overline{L}) \to \aH_{\anony}(X|Y,  \overline{M})
\]
given by
$s \mapsto s \otimes (\rest{t}{Y})$. Thus (1) follows.

\medskip
(2)
First let us see that
\addtocounter{Claim}{1}
\begin{equation}
\label{eqn:prop:basic:properties:rest:vol:1}
\left| \avol_{\anony}\left(X|Y, \overline{L} + \overline{\OO}(\lambda) \right) - \avol_{\anony}\left(X|Y, \overline{L}\right) \right| \leq
d' \vol(X_{\QQ}|Y_{\QQ}, L_{\QQ}) \vert \lambda  \vert.
\end{equation}
for any $\overline{L} \in \aPic(X;C^0)$ and $\lambda \in \RR$.
Without loss of generality, we may assume that $\lambda \geq 0$. As the assignment is proper,
for each $m \geq 1$, there is a symmetric and bounded convex set $\Delta_m$ such that
\[
 \aH_{\anony}(X|Y, m\overline{L} + \overline{\OO}(\mu)) = H^0(X|Y, mL) \cap \exp(\mu) \Delta_m
\]
for all $\mu \in \RR$. Thus, by using Lemma~\ref{lem:conv:module:rest:estimate},
\begin{multline*}
0 \leq \log \# \aH_{\anony}(X|Y, m(\overline{L} + \overline{\OO}(\lambda)))  - \log \# \aH_{\anony}(X|Y, m\overline{L}) \\
= \log \# (H^0(X|Y, m L) \cap \exp(m \lambda) \Delta_m) - \log \# (H^0(X|Y, L) \cap \Delta_m)\\
\leq  \log (\lceil 2 \exp(m \lambda) \rceil) \dim_{\QQ} H^0(X_{\QQ}|Y_{\QQ}, m L_{\QQ}),
\end{multline*}
which implies \eqref{eqn:prop:basic:properties:rest:vol:1}.

For $f \in C^0(X)$, if we set $\lambda = \Vert f \Vert_{\sup}$, then $-\lambda \leq f \leq \lambda$.
Thus the proposition follows from \eqref{eqn:prop:basic:properties:rest:vol:1}.
\end{proof}

The following theorem is the main result of this section.

\begin{Theorem}
\label{thm:vol:arith:rest:linear:system}
\begin{enumerate}
\renewcommand{\labelenumi}{\rom{(\arabic{enumi})}}
\item
If $\overline{L}$ is $Y$-big, then
\[
\avol_{\anony}\left(X|Y, \overline{L}\right) = \lim_{m\to\infty} \frac{\log \# \aH_{\anony}(X|Y, m\overline{L})}{m^{d'}/d'!}.
\]
In particular, if $\overline{L}$ is $Y$-big, then
$\avol_{\anony}\left(X|Y, n \overline{L}\right) = n^{d'}\avol_{\anony}\left(X|Y, \overline{L}\right)$ for all non-negative integers $n$.

\item
If $\overline{L}$ and $\overline{M}$ are $Y$-big continuous hermitian invertible sheaves on $X$, then
\[
\avol_{\anony}\left(X|Y, \overline{L}+\overline{M}\right)^{\frac{1}{d'}}
\geq \avol_{\anony}\left(X|Y, \overline{L}\right)^{\frac{1}{d'}} +
\avol_{\anony}\left(X|Y, \overline{M}\right)^{\frac{1}{d'}}.
\]

\item
If $\overline{L}$ is $Y$-big, then,
for any positive $\epsilon$, there is a positive integer $n_0 = n_0(\epsilon)$ such that,
for all $n \geq n_0$,
\[
\liminf_{k\to\infty} \frac{
\log\#(K_{k, n})}{n^{d'}k^{d'}}
\geq \frac{\avol_{\anony}(X|Y,\overline{L})}{d'!} - \epsilon,
\]
where $K_{k, n}$ is the convex lattice hull of
\[
V_{k,n} = \{ s_1 \otimes \cdots \otimes s_k \mid s_1, \ldots, s_k \in \aH_{\anony}(X|Y, n \overline{L})\}
\]
in $H^0(X|Y, knL)$ .
\end{enumerate}
\end{Theorem}

\begin{proof}
Let $Z_{\cdot} : Z_0 = Y \supset Z_1 \supset Z_2 \supset \cdots \supset Z_{d'}$ be a good flag over a prime $p$ on $Y$.

\medskip
(1)
Let $\Delta$ be the closure of
\[
\bigcup_{m=1}^{\infty} \frac{1}{m} \nu_{Z_{\cdot}}(\aH_{\anony}(X|Y, m\overline{L}) \setminus \{ 0 \})
\]
in $\RR^{d'}$.
Then, by Proposition~\ref{prop:Y:big:span:Okounkov:body}, \cite[Lemma~2.4]{YuanVol} and \cite[Proposition~2.1]{LazMus},
\[
\vol(\Delta) = \lim_{m\to\infty} \frac{ \#\nu_{Z_{\cdot}}(\aH_{\anony}(X|Y, m\overline{L}) \setminus \{ 0 \})}{m^{d'}}.
\]
By Corollary~\ref{thm:main:cor},
there is a constant $c$ depending only on $\overline{L}$ such that
\begin{multline*}
 \nu_{Z_{\cdot}}(\aH_{\anony}(X|Y, m\overline{L}) \setminus \{ 0 \}) \log p -  \frac{cm^{d'}}{\log p}
 \leq  \log \#\aH_{\anony}(X|Y, m\overline{L})  \\
 \leq \nu_{Z_{\cdot}}(\aH_{\anony}(X|Y, m\overline{L}) \setminus \{ 0 \}) \log p + \frac{cm^{d'}}{\log p},
\end{multline*}
which implies that
\begin{multline*}
\vol(\Delta) \log p -  \frac{c}{\log p} \leq
\liminf_{m\to\infty} \frac{\log \#\aH_{\anony}(X|Y, m\overline{L})}{m^{d'}} \\
\leq \limsup_{m\to\infty} \frac{\log \#\aH_{\anony}(X|Y, m\overline{L})}{m^{d'}}
\leq \vol(\Delta) \log p +  \frac{c}{\log p}.
\end{multline*}
Hence
\[
\limsup_{m\to\infty} \frac{\log \#\aH_{\anony}(X|Y, m\overline{L})}{m^{d'}} - \liminf_{m\to\infty} \frac{\log \#\aH_{\anony}(X|Y, m\overline{L})}{m^{d'}} \leq \frac{2c}{\log p}.
\]
Thus, as $p$ goes to $\infty$,
we have
\[
\limsup_{m\to\infty} \frac{\log \#\aH_{\anony}(X|Y, m\overline{L})}{m^{d'}} = \liminf_{m\to\infty} \frac{\log \#\aH_{\anony}(X|Y, m\overline{L})}{m^{d'}}.
\]
Moreover, we can see that
\addtocounter{Claim}{1}
\begin{equation}
\label{eqn:thm:vol:arith:rest:linear:system:1}
\left|
\avol_{\anony}\left(X|Y, \overline{L}\right) - \vol(\Delta) d'! \log p \right|
\leq \frac{cd'!}{\log p}.
\end{equation}

\medskip
(2)
Let $\Delta'$ and $\Delta''$ be the closure of
\[
\bigcup_{m=1}^{\infty} \frac{1}{m} \nu_{Z_{\cdot}}(\aH_{\anony}(X|Y, m\overline{M}) \setminus \{ 0 \})\quad\text{and}\quad
\bigcup_{m=1}^{\infty} \frac{1}{m} \nu_{Z_{\cdot}}(\aH_{\anony}(X|Y, m(\overline{L}+\overline{M})) \setminus \{ 0 \})
\]
in $\RR^{d'}$.
Since 
\begin{multline*}
\nu_{Z_{\cdot}}(\aH_{\anony}(X|Y, m\overline{L}) \setminus \{ 0 \}) + \nu_{Z_{\cdot}}(\aH_{\anony}(X|Y, m\overline{M}) \setminus \{ 0 \}) \\
=
\{ \nu_{Z_{\cdot}}( s \otimes s') \mid s \in \aH_{\anony}(X|Y, m\overline{L})\setminus\{ 0 \},\  s' \in \aH_{\anony}(X|Y, m\overline{M}) \setminus \{ 0 \} \} \\
\subseteq \nu_{Z_{\cdot}}(\aH_{\anony}(X|Y, m(\overline{L}+\overline{M})) \setminus \{ 0 \}),
\end{multline*}
we have $\Delta + \Delta' \subseteq \Delta''$.
Thus, by Brunn-Minkowski's theorem,
\[
\vol(\Delta'')^{\frac{1}{d'}} \geq \vol(\Delta + \Delta')^{\frac{1}{d'}}  \geq \vol(\Delta)^{\frac{1}{d'}} 
+ \vol(\Delta')^{\frac{1}{d'}}.
\]
Note that \eqref{eqn:thm:vol:arith:rest:linear:system:1} also holds for $\overline{L}$ and
$\overline{L} + \overline{M}$ with another constants $c'$ and $c''$.
Hence, for a small positive number $\epsilon$, if $p$ is a sufficiently large prime number, then
\begin{align*}
&\left|
\avol_{\anony}\left(X|Y, \overline{L}\right)- \vol(\Delta) d'! \log p \right| \leq \epsilon, \\
&\left|
\avol_{\anony}\left(X|Y, \overline{M}\right) - \vol(\Delta') d'! \log p \right| \leq \epsilon \quad\text{and}\\
&\left|
\avol_{\anony}\left(X|Y, \overline{L}+\overline{M}\right)- \vol(\Delta'') d'! \log p \right| \leq \epsilon
\end{align*}
hold. Therefore,
\[
\left(\avol_{\anony}\left(X|Y, \overline{L}+\overline{M}\right)+ \epsilon  \right)^{\frac{1}{d'}}
\geq \left( \avol_{\anony}\left(X|Y, \overline{L}\right)  - \epsilon \right)^{\frac{1}{d'}}
+ \left( \avol_{\anony}\left(X|Y, \overline{M}\right) - \epsilon \right)^{\frac{1}{d'}},
\]
as required.

\medskip
(3)
Let $c$ be a constant for $Y$ and $\rest{\overline{L}}{Y}$
as in Corollary~\ref{thm:main:cor}.
We choose a good flag 
$Z_{\cdot} : Z_0 = Y \supset Z_1 \supset Z_2 \supset \cdots \supset Z_{d'}$
over a prime $p$ with $c/(\log p) \leq \epsilon/3$.
Let $\epsilon'$ be a positive number with $\epsilon' \log p \leq \epsilon/3$.
By \cite[Proposition~3.1]{LazMus},
there is a positive integer $n_0$ such that
\[
\lim_{k \to \infty} \frac{\#(k \ast \nu(\aH_{\anony}(X|Y, n\overline{L})\setminus \{ 0 \}))}{k^{d'}n^{d'}} \geq
\vol(\Delta) - \epsilon'
\]
for all $n \geq n_0$.
Note that 
\[
\nu(K_{k,n}\setminus \{ 0 \}) \supseteq k \ast \nu(\aH_{\anony}(X|Y,n\overline{L})\setminus \{ 0 \})
\]
and
\[
 \log \#(K_{k,n}) \geq
 \# \nu(K_{k,n}\setminus \{ 0 \}) \log p - (\epsilon/3)k^{d'}n^{d'}
\]
by Corollary~\ref{thm:main:cor} for $k \gg 1$.
Thus
\[
\frac{ \log \#(K_{k,n})}{k^{d'}n^{d'}} \geq
\frac{\#(k \ast \nu(\aH_{\anony}(X|Y,n\overline{L})\setminus \{ 0 \}))\log p}{k^{d'}n^{d'}} - \epsilon/3,
\]
which implies that
\begin{align*}
\liminf_{k \to\infty} \frac{ \log \#(K_{k,n})}{k^{d'}n^{d'}}
& \geq 
\lim_{k\to\infty} \frac{\#(k \ast \nu(\aH_{\anony}(X|Y,n\overline{L})\setminus \{ 0 \}))\log p}{k^{d'}n^{d'}} - \epsilon/3 \\
& \geq (\vol(\Delta) - \epsilon')\log p - \epsilon/3 \\
& \geq \vol(\Delta) \log p - 2\epsilon/3.
\end{align*}
Moreover, by \eqref{eqn:thm:vol:arith:rest:linear:system:1},
\[
\vol(\Delta) \log p \geq \frac{\avol_{\anony}(X|Y,\overline{L})}{d'!} - \epsilon/3.
\]
Thus we obtain (3).
\end{proof}

In the remaining of this section, let us consider consequences of Theorem~\ref{thm:vol:arith:rest:linear:system}.

\begin{Corollary}
\label{cor:thm:vol:arith:rest:linear:system:otimes:R}
There is a unique 
continuous function 
\[
\avol'_{\anony}(X|Y,-) : \aBig_{\otimes \RR}(X;Y) \to \RR
\]
with the following properties:
\begin{enumerate}
\renewcommand{\labelenumi}{\rom{(\arabic{enumi})}}
\item
$\avol'_{\anony}(X|Y,\iota(\overline{L})) = \avol_{\anony}(X|Y,\overline{L})$
holds for all $\overline{L} \in \aBig(X;Y)$.

\item
$\avol'_{\anony}(X|Y,\lambda x) = \lambda^{d'} \avol'_{\anony}(X|Y,x)$
holds
for all $\lambda \in \RR_{>0}$ and $x \in \aBig_{\otimes \RR}(X;Y)$.

\item
$\avol'_{\anony}\left(X|Y, x + y\right)^{\frac{1}{d'}}
\geq \avol'_{\anony}\left(X|Y, x \right)^{\frac{1}{d'}} +
\avol'_{\anony}\left(X|Y, y \right)^{\frac{1}{d'}}$ holds
for all $x, y \in \aBig_{\otimes \RR}(X;Y)$.
\end{enumerate}
\end{Corollary}

\begin{proof}
It follows from Theorem~\ref{thm:vol:arith:rest:linear:system}, Proposition~\ref{prop:QSat} and
Corollary~\ref{cor:concave:extension}.
\end{proof}

\begin{Corollary}
\label{cor:thm:vol:arith:rest:linear:system:R}
If the assignment $\overline{L} \mapsto \aH_{\anony}(X|Y, \overline{L})$ is proper, then
there is a unique 
continuous function 
\[
\avol''_{\anony}(X|Y,-) : \aBig_{\RR}(X;Y) \to \RR
\]
with the following properties:
\begin{enumerate}
\renewcommand{\labelenumi}{\rom{(\arabic{enumi})}}
\item
$\avol''_{\anony}(X|Y,\pi(x')) = \avol'_{\anony}(X|Y,x')$
holds for all $x' \in \aBig_{\otimes \RR}(X;Y)$.

\item
$\avol''_{\anony}(X|Y,\lambda x) = \lambda^{d'} \avol''_{\anony}(X|Y,x)$
holds
for all $\lambda \in \RR_{>0}$ and $x \in \aBig_{\RR}(X;Y)$.

\item
$\avol''_{\anony}\left(X|Y, x + y\right)^{\frac{1}{d'}}
\geq \avol''_{\anony}\left(X|Y, x\right)^{\frac{1}{d'}} +
\avol''_{\anony}\left(X|Y, y\right)^{\frac{1}{d'}}$ holds
for all $x, y \in \aBig_{\RR}(X;Y)$.
\end{enumerate}
\end{Corollary}

\begin{proof}
Let us begin with the following estimation.
\addtocounter{Claim}{1}
\begin{equation}
\label{eqn:cor:thm:vol:arith:rest:linear:system:R:1}
\left| \avol'_{\anony}\left(X|Y, \overline{L} + \overline{\OO}(f) \right) - \avol'_{\anony}\left(X|Y, \overline{L}\right) \right| \leq
d' \vol(X_{\QQ}|Y_{\QQ}, L_{\QQ}) \Vert f \Vert_{\sup}.
\end{equation}
for any $\overline{L} \in \aBig_{\otimes \RR}(X;Y)$ and $f \in C^0(X)$ with
$\overline{L} + \OO(f) \in \aBig_{\otimes \RR}(X;Y)$.
By using the openness of $\aBig_{\otimes \RR}(X;Y)$ and the continuity of
$\avol'_{\anony}(X|Y, - )$ on $\aBig_{\otimes \RR}(X;Y)$, it is sufficient to see \eqref{eqn:cor:thm:vol:arith:rest:linear:system:R:1} for
$\overline{L} \in \aBig_{\QQ}(X;Y)$.
Thus $\overline{L} = (1/n)\iota(\overline{M})$ for some $\overline{M} \in \aBig(X;Y)$ and $n \in \ZZ_{>0}$, and hence,
by Proposition~\ref{prop:basic:properties:rest:vol},
\begin{multline*}
\left| \avol'_{\anony}\left(X|Y, \overline{L} + \overline{\OO}(f) \right) - \avol'_{\anony}\left(X|Y, \overline{L}\right) \right| \\
=\left| \avol'_{\anony}\left(X|Y, (1/n) \iota(\overline{M} + \overline{\OO}(nf)) \right) - \avol'_{\anony}\left(X|Y, (1/n) \iota(\overline{M})\right) \right| \\
=(1/n)^{d'} \left| \avol_{\anony}\left(X|Y, \overline{M} + \overline{\OO}(nf) \right) - \avol_{\anony}\left(X|Y, \overline{M}\right) \right| \\
= (1/n)^{d'} d' \vol(X_{\QQ}|Y_{\QQ}, M_{\QQ})\Vert n f \Vert_{\sup} = d' \vol(X_{\QQ}|Y_{\QQ}, L_{\QQ})\Vert  f \Vert_{\sup}.
\end{multline*}

Let us observe that there is a function 
\[
\avol''_{\anony}(X|Y,-) : \aBig_{\RR}(X;Y) \to \RR
\]
such that the following diagram is commutative:
\[
\xymatrix{
\aBig_{\otimes \RR}(X;Y)  \ar[rr]^(.6){\avol'_{\anony}(X|Y,-)} \ar[d]_{\pi} & & \RR \\
\aBig_{\RR}(X;Y) \ar[rru]_(.5){\ \ \avol''_{\anony}(X|Y-)} & \\
}
\]
Namely, we need to show that if
$\pi(x') = \pi(y')$ for $x', y' \in \aBig_{\otimes \RR}(X;Y)$,
then 
\[
\avol'_{\anony}\left(X|Y, x' \right) = \avol'_{\anony}\left(X|Y, y' \right).
\]
As $\pi(x') = \pi(y')$, there is $z \in N(X)$ such that $y' = x' + z$.
We set $z = \overline{\OO}(f_1)\otimes a_1 + \cdots + \overline{\OO}(f_r) \otimes a_r$ with
$a_1 f_1 + \cdots + a_r f_r = 0$, where $f_1, \ldots, f_r \in C^0(X)$ and
$a_1, \ldots, a_r \in \RR$.
Let us take a sequence $\{ a_{in} \}_{n=1}^{\infty}$ in $\QQ$ such that
$a_i = \lim_{n\to\infty} a_{in}$. We set $\phi_n = a_{1n} f_1 + \cdots + a_{rn}f_r$.
Then
\begin{multline*}
\Vert \phi_n \Vert_{\sup} = \Vert (a_{1n} - a_1) f_1 + \cdots + (a_{rn}-a_r) f_r \Vert_{\sup} \\
\leq
\vert a_{1n} - a_1\vert \Vert f_1  \Vert_{\sup} + \cdots + \vert a_{rn}-a_r\vert  \Vert f_r \Vert_{\sup}.
\end{multline*}
Thus $\lim_{n\to\infty} \Vert \phi_n \Vert_{\sup} = 0$.
If we put $z_n = \overline{\OO}(f_1)\otimes a_{1n} + \cdots + \overline{\OO}(f_r) \otimes a_{rn}$,
then $\lim_{n\to\infty} z_n = z$ in $\aPic_{\otimes \RR}(X;C^0)$ and $z_n = \overline{\OO}(\phi_n)$ in $\aPic_{\QQ}(X;C^0)$.
Thus, by \eqref{eqn:cor:thm:vol:arith:rest:linear:system:R:1},
\begin{multline*}
\left| \avol'_{\anony}\left(X|Y, x' + z_n \right) - \avol'_{\anony}\left(X|Y, x'\right) \right| \\
= \left| \avol'_{\anony}\left(X|Y, x' + \overline{\OO}(\phi_n) \right) - \avol'_{\anony}\left(X|Y, x' \right) \right| \\
\leq d' \vol(X_{\QQ}|Y_{\QQ}, x'_{\QQ}) \Vert \phi_n \Vert_{\sup}
\end{multline*}
for $n\gg 1$. Therefore, as $n$ goes to $\infty$,
$\avol'_{\anony}\left(X|Y, y'\right) = \avol'_{\anony}\left(X|Y, x'\right)$.

The properties (2) and (3) are obvious.
The continuity of $\avol''_{\anony}(X|Y,-)$ follows from that $\pi$ is an open map.
\end{proof}

\section{Restricted volume for ample $C^{\infty}$-hermitian invertible sheaf}

In this section, let us consider the restricted volume for an ample $C^{\infty}$-hermitian invertible sheaf and observe several consequences.

Let $(M, \Vert\cdot\Vert)$ be a normed $\ZZ$-module, that is,
$M$ is a finitely generated $\ZZ$-module and $\Vert\cdot\Vert$ is a norm of $M \otimes_{\ZZ} \RR$.
We assume that $M$ is free.
According as \cite{ZhPos},
we define $\lambda(M,\Vert\cdot\Vert)$ and $\lambda'(M,\Vert\cdot\Vert)$ as follows:
\begin{align*}
\lambda(M,\Vert\cdot\Vert) & := \inf \left\{ \lambda \in \RR_{>0} \ \left|\ 
\begin{array}{l}
\text{there is a free basis $e_1, \ldots, e_n$ of $M$ such that } \\
\text{$\Vert e_i \Vert \leq \lambda$ for all $i$} 
\end{array}
\right\}\right.,\\
\lambda'(M,\Vert\cdot\Vert) & := \inf \left\{ \lambda' \in \RR_{>0} \ \left|\ 
\begin{array}{l}
\text{there are $e_1, \ldots, e_n \in M$ such that $e_1, \ldots, e_n $} \\
\text{form a basis over $\QQ$ and $\Vert e_i \Vert \leq \lambda'$ for all $i$} 
\end{array}
\right\}\right..
\end{align*}
By \cite[Lemma~1.7]{ZhPS} and \cite[(4.3.1)]{ZhPos}, it is known that
\renewcommand{\theequation}{\arabic{section}.\arabic{Theorem}}
\addtocounter{Theorem}{1}
\begin{equation}
\label{eqn:Zhang:lambda:inequality}
\lambda'(M,\Vert\cdot\Vert) \leq \lambda(M,\Vert\cdot\Vert) \leq \lambda'(M,\Vert\cdot\Vert) \rank M.
\end{equation}
\renewcommand{\theequation}{\arabic{section}.\arabic{Theorem}.\arabic{Claim}}
Let us begin with the following lemma.

\begin{Lemma}
\label{lem:small:sec:ideal}
Let $X$ be a projective  arithmetic variety, $\overline{A}$ an ample $C^{\infty}$-hermitian invertible sheaf on $X$ and 
$I$ an ideal sheaf of $\OO_X$. Then there are a positive integer $n_0$ and a positive number $\epsilon_0$ with the following property:
for all $n \geq n_0$, we can find a free basis $e_1, \ldots, e_N$ of $H^0(X, nA \otimes I)$ as a $\ZZ$-module such that
$\Vert e_i \Vert_{\sup} \leq e^{-n \epsilon_0}$ for all $i$, where the norm $\Vert \cdot \Vert_{\sup}$ of $H^0(X, nA \otimes I)$
is the sub-norm induced by the inclusion map $H^0(X, nA \otimes I) \hookrightarrow H^0(X, nA)$ and the sup norm of  $H^0(X, nA)$.
\end{Lemma}

\begin{proof}
Since $\overline{A}$ is ample, there is a positive integer $n_1$ such that
$H^0(X, n_1A)$ is generated by sections $s$ with $\Vert s \Vert_{\sup} < 1$, and that
$R=\bigoplus_{m \geq 0} H^0(X, mn_1A)$ is generated by $H^0(X, n_1A)$.
Let $s_1, \ldots, s_r$ be non-zero generators of $H^0(X, n_1A)$ with $\Vert s_i \Vert_{\sup} < 1$.
Since $\bigoplus_{n \geq 0} H^0(X, nA \otimes I)$ is finitely generated as
a $R$-module, we can find $m_1, \ldots, m_s$ such that
$m_i \in H^0(X, a_i A \otimes I)$ and
$m_1, \ldots, m_s$ generate $\bigoplus_{n \geq 0} H^0(X, nA \otimes I)$ as a $R$-module.
We set
\[
\begin{cases}
e^{-c_1} = \max \{ \Vert s_1 \Vert_{\sup}, \ldots, \Vert s_r \Vert_{\sup} \}, \\
e^{c_2} = \max \{ \Vert m_1 \Vert_{\sup}, \ldots, \Vert m_s\Vert_{\sup} \}.
\end{cases}
\]
As $\Vert s_i \Vert_{\sup} < 1$ for $i=1, \ldots, r$, we have $c_1 > 0$.
By our construction, $H^0(X, nA \otimes I)$ is generated by elements of the form $s_1^{k_1} \otimes \cdots \otimes s_r^{k_r} \otimes m_i$
with $n_1(k_1 + \cdots + k_r) + a_i = n$ and $(k_1, \ldots, k_r) \in (\ZZ_{\geq 0})^r$. On the other hand,
\begin{multline*}
\Vert s_1^{k_1} \otimes \cdots \otimes s_r^{k_r} \otimes m_i \Vert_{\sup}
\leq \Vert s_1 \Vert_{\sup}^{k_1} \cdots \Vert s_r\Vert_{\sup}^{k_r} \Vert m_i \Vert_{\sup}  \\
\leq \exp(-c_1(k_1 + \cdots + k_r) + c_2)
\leq \exp(-(c_1/n_1) n + c_3),
\end{multline*}
where $c_3 = \max_{i=1, \ldots, s} \{ (c_1/n_1)a_i + c_2 \}$, which means that
\[
\lambda'(H^0(X, nA \otimes I), \Vert\cdot\Vert_{\sup}) \leq \exp(-(c_1/n_1) n + c_3).
\]
Note that there is a positive number $\epsilon_0$ such that
\[
\rank H^0(X, nA) \exp(-(c_1/n_1) n + c_3) \leq \exp(-\epsilon_0 n)
\]
for $n \gg 1$.
Thus the lemma follows from \eqref{eqn:Zhang:lambda:inequality}.
\end{proof}

\begin{Theorem}
\label{thm:rest:ample:vol}
Let $X$ be a projective arithmetic variety and $Y$ a $d'$-dimensional  arithmetic subvariety of $X$.
\begin{enumerate}
\renewcommand{\labelenumi}{\rom{(\arabic{enumi})}}
\item
If $\overline{A}$ is an ample $C^{\infty}$-hermitian invertible sheaf on $X$, then
\[
\lim_{m\to\infty} \frac{\log \# \Image(\aH(X, m\overline{A}) \to H^0(X|Y, mA))}{m^{d'}/d'!} = 
\avol_{\quot}(X|Y, \overline{A}) > 0.
\]

\item
If $X$ is generically smooth and
$\overline{A}$ is an ample $C^{\infty}$-hermitian invertible sheaf on $X$, then
\[
\avol_{\quot}(X|Y, \overline{A}) = \adeg \left( \acherncl_1(\rest{\overline{A}}{Y})^{d'} \right).
\]
\end{enumerate}
\end{Theorem}

\begin{proof}
(1)
First let us see 
\[
\lim_{m\to\infty} \frac{\log \# \Image(\aH(X, m\overline{A}) \to H^0(X|Y, mA))}{m^{d'}/d'!} = 
\avol_{\quot}(X|Y, \overline{A}).
\]
Let $I$ be  the defining ideal sheaf of $Y$.
Let $\epsilon$ be an arbitrary positive number such that $\overline{A} - \overline{\OO}(\epsilon)$ is ample. 

\begin{Claim}
$\aH_{\quot}(X|Y, m(\overline{A} - \overline{\OO}(\epsilon))) \subseteq
\Image(\aH(X, m\overline{A}) \to H^0(X|Y, mA))$
for $m \gg 1$.
\end{Claim}

By Lemma~\ref{lem:small:sec:ideal}, there is a positive number $\epsilon_0$ such that,
if $m \gg 1$, then we can find a free basis $e_1, \ldots, e_N$ of $H^0(X, m A \otimes I)$ such that
$\Vert e_i \Vert_{\sup} \leq e^{-\epsilon_0 m}$ for all $i$. 
We choose $e_{N+1}, \ldots, e_M \in H^0(X, mA)$ such that
$\rest{e_{N+1}}{Y}, \ldots, \rest{e_M}{Y}$ form a free basis of $H^0(X|Y, mA)$.
Then $e_1, \ldots, e_M$ form a free basis of $H^0(X, mA)$.
Let $s \in \aH_{\quot}(X|Y, m(\overline{A} - \overline{\OO}(\epsilon)))$.
Then there is $s' \in H^0(X, mA) \otimes \RR$ such that
$\rest{s'}{Y} = s$ and $\Vert s' \Vert_{\sup} = \Vert s \Vert_{\sup,\quot}^{X|Y} \leq e^{-\epsilon m}$.
We set $s' = \sum_{i=1}^M c_i e_i$ ($c_i \in \RR$).
Since
\[
\rest{s'}{Y} = \sum_{i=N+1}^M  \rest{c_i e_{i}}{Y} = s \in H^0(X|Y, mA),
\]
we have $c_i \in \ZZ$ for all $i=N+1, \ldots, M$.
Here we put 
\[
\tilde{s} = \sum_{i=1}^N \lceil c_i \rceil e_i + \sum_{i=N+1}^M c_i e_i.
\]
Then $\rest{\tilde{s}}{Y} = s$ and
\[
\Vert \tilde{s} \Vert_{\sup} = \left\Vert s' + \sum_{i=1}^N ( \lceil c_i \rceil - c_i)e_i \right\Vert_{\sup} \leq e^{-\epsilon m}  +
e^{-\epsilon_0 m} \rank H^0(X, mA),
\]
which means that, if $m \gg 1$, then $\tilde{s} \in \aH(X, m\overline{A})$.
Therefore, $s \in \Image(\aH(X, m\overline{A}) \to H^0(X|Y, mA))$.
\CQED

By the above claim,
\begin{multline*}
\avol_{\quot}(X|Y, \overline{A} - \overline{\OO}(\epsilon)) \leq
\liminf_{m\to\infty} \frac{\log \# \Image(\aH(X, m\overline{A}) \to H^0(X|Y, mA))}{m^{d'}/d'!} \\
\leq
\limsup_{m\to\infty} \frac{\log \# \Image(\aH(X, m\overline{A}) \to H^0(X|Y, mA))}{m^{d'}/d'!} \leq
\avol_{\quot}(X|Y, \overline{A}).
\end{multline*}
Hence the assertion follows because, by Proposition~\ref{prop:basic:properties:rest:vol},
\[
\avol_{\quot}(X|Y, \overline{A} - \overline{\OO}(\epsilon)) \geq 
\avol_{\quot}(X|Y, \overline{A})  - d'\epsilon\vol(X_{\QQ}|Y_{\QQ}, A_{\QQ})
\]
and $\epsilon$ is arbitrary small number.

\smallskip
Next we observe 
\[
\lim_{m\to\infty} \frac{\log \# \Image(\aH(X, m\overline{A}) \to H^0(X|Y, mA))}{m^{d'}/d'!} > 0.
\]
Let us choose a sufficiently large integer $n_0$ with the following properties:
\begin{enumerate}
\renewcommand{\labelenumi}{\rom{(\alph{enumi})}}
\item
$H^0(X, n_0A)$ has a free basis $\Sigma$ consisting of strictly small sections.

\item
$\Sym^m(H^0(X, n_0A)) \to H^0(X, mn_0A)$ is surjective for all $m \geq 1$.

\item
$H^0(X, mn_0A) \to H^0(Y, \rest{mn_0A}{Y})$ is surjective for all $m \geq 1$.
\end{enumerate}
We set $e^{-c} = \max \{ \Vert s \Vert_{\sup} \mid s \in \Sigma \}$. Then $c > 0$.
Moreover, we put
\[
\Sigma_m = \{ s_1 \otimes \cdots \otimes s_m \mid s_1, \ldots, s_m \in \Sigma \}.
\]
Note that $\Sigma_m$ generates $H^0(X, mn_0A)$ as a $\ZZ$-module 
and that $\Vert s \Vert_{\sup} \leq e^{-mc}$ for all $s \in \Sigma_m$.
Let $r_m$ be the rank of $H^0(Y, \rest{mn_0A}{Y})$.
Since $\{ \rest{s}{Y} \mid s \in \Sigma_m \}$ gives rise to a generator of $H^0(Y, \rest{mn_0A}{Y})$,
we can find $s_1, \ldots, s_{r_m} \in \Sigma_m$ such that
$\{ \rest{s_1}{Y}, \ldots, \rest{s_{r_m}}{Y} \}$ forms a basis of $H^0(Y, \rest{mn_0A}{Y}) \otimes {\QQ}$.
We put 
\[
S_m = \{ (a_1, \ldots, a_{r_m}) \in \ZZ^{r_m} \mid 0 \leq a_i \leq e^{cm}/r_m \}.
\]
Then the map $S_m \to H^0(Y, \rest{mn_0A}{Y})$ given by
\[
(a_1, \ldots, a_{r_m}) \mapsto a_1 \rest{s_1}{Y} + \cdots + a_{r_m} \rest{s_{r_m}}{Y}
\]
is injective.
Moreover, for $(a_1, \ldots, a_{r_m}) \in S_m$,
\[
\left\Vert \sum_{i=1}^{r_m} a_is_i  \right\Vert_{\sup} \leq \sum_{i=1}^{r_m} a_i \Vert s_i \Vert_{\sup} \leq \sum_{i=1}^{r_m} (e^{cm}/r_m) e^{-cm} = 1.
\]
Hence
\[
\#\Image(\aH(X, mn_0\overline{A}) \to H^0(X|Y, mn_0A))
\geq \#(S_m) \geq (e^{cm}/r_m)^{r_m}.
\]
Thus the second assertion follows. 

\medskip
(2) 
It is sufficient to show that $\avol_{\quot}(X|Y, \overline{A}) = \avol(Y, \rest{\overline{A}}{Y})$.
Let $\epsilon$ be an arbitrary positive number such that $\overline{A} - \overline{\OO}(\epsilon)$ is ample.
By (3) in Theorem~\ref{thm:vol:arith:rest:linear:system},
there is a positive number $n_1$ such that if we set $\aH(Y, n_1(\overline{A} - \overline{\OO}(\epsilon))) = \{ s_1, \ldots, s_l \}$,
then
\[
\liminf_{m\to\infty} \frac{\log \# \CL(\{ s_1^{a_1} \otimes \cdots \otimes s_l^{a_l} \mid (a_1, \ldots, a_l) \in \Gamma_m \})}{m^{d'}n_1^{d'}/d'!}
\geq \avol(Y, \rest{\overline{A} - \overline{\OO}(\epsilon)}{Y}) - \epsilon,
\]
where $\Gamma_m = \{ (a_1, \ldots, a_l) \in (\ZZ_{\geq 0})^l \mid a_1 + \cdots + a_l = m \}$.
Note that $\Vert s_i \Vert_{\sup} \leq e^{-n_1\epsilon}$ for all $i$.
By \cite[Theorem~2.2]{ZhPos}, if $m \gg 1$, then, for any $a_1, \ldots, a_l \in \ZZ_{\geq 0}$ with $\ a_1 + \cdots + a_l = m$,
there is $s(a_1, \ldots, a_l) \in H^0(X, mn_1A) \otimes \RR$ such that
$\rest{s(a_1, \ldots, a_l)}{Y} = s_1^{a_1} \otimes \cdots \otimes s_l^{a_l}$ and
\[
\Vert s(a_1, \ldots, a_l) \Vert_{\sup} \leq e^{m\epsilon}  \Vert s_1\Vert_{\sup}^{a_1} \cdots \Vert s_l\Vert_{\sup}^{a_l}
\leq e^{-\epsilon m (n_1 - 1)} < 1,
\]
which means that $s_1^{a_1} \otimes \cdots \otimes s_l^{a_l} \in \aH_{\quot}(X|Y, mn_1\overline{A})$.
Therefore, 
\[
 \CL(\{ s_1^{a_1} \otimes \cdots \otimes s_l^{a_l} \mid (a_1, \ldots, a_l) \in \Gamma_m \}) \subseteq \aH_{\quot}(X|Y, mn_1\overline{A}).
\]
Hence
\begin{align*}
\avol(Y, \rest{\overline{A}}{Y}) & \geq \avol_{\quot}(X|Y, \overline{A}) \\
& = \lim_{m\to\infty} \frac{\log \# \aH_{\quot}(X|Y, mn_1\overline{A})}{(mn_1)^{d'}/d'!} \\
& \geq \liminf_{m\to\infty} \frac{\log \# \CL(\{ s_1^{a_1} \otimes \cdots \otimes s_l^{a_l} \mid (a_1, \ldots, a_l) \in \Gamma_m \})}{m^{d'}n_1^{d'}/d'!} \\
& \geq  \avol(Y, \rest{\overline{A} - \overline{\OO}(\epsilon)}{Y}) - \epsilon \geq \avol(Y, \rest{A}{Y}) - \epsilon(d'\vol(Y_{\QQ}, A_{\QQ}) + 1),
\end{align*}
as required.
\end{proof}

\begin{Corollary}
Let $\overline{L} \mapsto \aH_{\anony}(X|Y, \overline{L})$ be an assignment of arithmetic restricted linear series from $X$ to $Y$.
Then we have the following.
\begin{enumerate}
\renewcommand{\labelenumi}{\rom{(\arabic{enumi})}}
\item
If $X$ is generically smooth and
$\overline{A}$ is an ample $C^{\infty}$-hermitian invertible sheaf on $X$,
then 
\[
\avol_{\anony}(X|Y, \overline{A}) = \avol(Y, \rest{\overline{A}}{Y}).
\]

\item
If $\overline{L}$ is a $Y$-big continuous hermitian invertible sheaf on $X$,
then 
\[
\avol_{\anony}(X|Y, \overline{L}) > 0.
\]

\item
If $x \in \aBig_{\otimes \RR}(X;Y)$, then $\avol_{\anony}(X|Y, x) > 0$.
\end{enumerate}
\end{Corollary}

\begin{proof}
(1) is a consequence of Theorem~\ref{thm:rest:ample:vol}.

(2)
As $\overline{L}$ is $Y$-big, there are a positive integer $n$ and an ample
$C^{\infty}$-hermitian invertible sheaf $\overline{A}$ on $X$ such that
$n \overline{L} \geq_Y \overline{A}$, so that, by (1) in Proposition~\ref{prop:basic:properties:rest:vol} and
(1) in Theorem~\ref{thm:rest:ample:vol},
\[
n^{d'} \avol_{\anony}(X|Y, \overline{L}) = \avol_{\anony}(X|Y, n \overline{L}) \geq \avol_{\anony}(X|Y, \overline{A}) > 0.
\]

(3) If $x \in \aBig_{\otimes \RR}(X;Y)$, there are positive numbers $a_1, \ldots, a_r$ and
$\overline{L}_1, \ldots, \overline{L}_r \in \aBig(X;Y)$ such that
$x = \overline{L}_1 \otimes a_1 + \cdots + \overline{L}_r \otimes a_r$. Hence, by (2)  and Corollary~\ref{cor:thm:vol:arith:rest:linear:system:otimes:R},
\begin{align*}
\avol_{\anony}(X|Y, x)^{\frac{1}{d'}} & \geq \avol_{\anony}(X|Y,  \overline{L}_1 \otimes a_1)^{\frac{1}{d'}} + \cdots + \avol_{\anony}(X|Y,  \overline{L}_r \otimes a_r)^{\frac{1}{d'}} \\
& \geq a_1 \avol_{\anony}(X|Y,  \overline{L}_1)^{\frac{1}{d'}} + \cdots + a_r \avol_{\anony}(X|Y,  \overline{L}_r)^{\frac{1}{d'}} > 0.
\end{align*}
\end{proof}

\end{document}